\documentclass[preprint]{siamltex}

\usepackage{lipsum}
\usepackage{amsfonts}
\usepackage{graphicx}
\usepackage{algorithmic}
\usepackage{bm}

\usepackage{multirow}
\newtheorem{remark}{Remark}[section]

\newtheorem{algorithm}{Algorithm}[section]

\title{Structure-preserving parametric finite element method for curve diffusion based on Lagrange multiplier approaches} 

\author{
Harald Garcke\thanks{Fakult\"at f\"ur Mathematik, Universit\"at Regensburg, Regensburg, Germany (\tt{harald.garcke@mathematik.uni-regensburg.de})}
	\and
	Wei Jiang\thanks{School of Mathematics and Statistics, Hubei Key Laboratory of Computational Science, Wuhan University, Wuhan 430072, China (\tt{jiangwei1007@whu.edu.cn}). This author's research was supported by the NSFC 12271414 and 11871384. }		
	\and Chunmei Su\thanks{Yau Mathematical Sciences Center, Tsinghua University, Beijing, 100084, China
		(\tt{sucm@tsinghua.edu.cn}). This author's research was supported by National Key R\&D Program of China (2023YFA1008902) and the National Natural Science Foundation of China No. 12201342.}
	\and Ganghui Zhang\thanks{Yau Mathematical Sciences Center, Tsinghua University, Beijing, 100084, China
		(\tt{gh-zhang19@mails.tsinghua.edu.cn}).}
}

\usepackage{amsopn}
\usepackage{mathrsfs}
\usepackage{amsfonts}
\usepackage{stmaryrd}
\usepackage{amsmath}
\usepackage{amssymb}
\usepackage{multirow}
\usepackage{tabularx}
\usepackage{booktabs}
\usepackage{cite}
\usepackage{verbatim}
\usepackage{color}
\usepackage{graphicx}
\usepackage{subfigure}
\usepackage{epsfig,epstopdf,color}
\allowdisplaybreaks[4]

\newcommand{\p}{\partial}

\renewcommand{\theequation}{\arabic{section}.\arabic{equation}}
\newcommand{\be}{\begin{equation}}
\newcommand{\ee}{\end{equation}}
\newcommand{\ba}{\begin{array}}
\newcommand{\ea}{\end{array}}
\newcommand{\bea}{\begin{eqnarray}}
\newcommand{\eea}{\end{eqnarray}}
\newcommand{\beas}{\begin{eqnarray*}}
\newcommand{\eeas}{\end{eqnarray*}}

\def\R{{\mathbb R}}
\def\I{{\mathbb I}}

\newcommand{\cE}{\mathcal E}

\renewcommand{\l}{\left}
\renewcommand{\r}{\right}

\renewcommand{\d}{\mathrm{d}}

\newcommand{\bn}{\mathbf{n}}

\newcommand{\bX}{\mathbf{X}}
\newcommand{\bM}{\mathbf{M}}
\newcommand{\bY}{\mathbf{Y}}

\begin{document}
\maketitle
\begin{abstract}
We propose a novel formulation for parametric finite element methods to simulate surface diffusion of closed curves, which is also called as the curve diffusion. Several high-order temporal discretizations are proposed based on this new formulation. To ensure that the numerical methods preserve geometric structures of curve diffusion (i.e., the perimeter-decreasing and area-preserving properties), our formulation incorporates two scalar Lagrange multipliers and two evolution equations involving the perimeter and area, respectively. By discretizing the spatial variable using piecewise linear finite elements  and the temporal variable using either the Crank-Nicolson method or the backward differentiation formulae method, we develop high-order temporal schemes that effectively preserve the structure at a fully discrete level. These new schemes are implicit and can be efficiently solved using Newton's method. Extensive numerical experiments demonstrate that our methods achieve the desired temporal accuracy, as measured by the manifold distance, while simultaneously preserving the geometric structure of the curve diffusion.

\end{abstract}



\section{Introduction}

The curve diffusion characterizes the interface dynamics of solid materials whose evolution is governed by mass diffusion within the interface \cite{Mullins,Davi}. This physical phenomenon has garnered significant interest in the fields of differential geometry \cite{Elliott-Garcke,Miura,Fusco} and materials  science \cite{Zhao2021,Wang,Li-Voigt}. It is well-known that curve diffusion exhibits  two key geometric structures  \cite{Bao-Zhao,Garcke2013}, perimeter reduction and area preservation. Therefore, from a numerical standpoint, it is essential to develop efficient methods that maintain these properties at the discrete level.

Numerous numerical methods as well as formulations have been proposed for simulating the curve diffusion or surface diffusion, including level set methods \cite{Chopp}, parametric formulations \cite{Bansch2005}, graph formulations \cite{Bansch2004,Dziuk2002} and formulations based on axisymmetric geometry \cite{DDE2003}. Among these, parametric finite element methods (PFEMs) are  particularly compelling. Building upon Dziuk's PFEM \cite{Dziuk1990} for simulating the mean curvature flow, B\"ansch, Morin and Nochetto \cite{Bansch2005} proposed a PFEM for curve diffusion that requires mesh regularization to avoid mesh distortion. Subsequently, Barrett, Garcke and N\"urnberg \cite{BGN07A} introduced a novel formulation for curve diffusion that incorporates implicit tangential motions to maintain good mesh quality for long-time evolution. They later extended this approach to three-dimensional surface diffusion \cite{BGN08B}. Recently, Hu and Li \cite{Hu-Li} employed the idea of  artificial tangential velocity to design an evolving surface finite element method and proved its convergence for finite elements with order $k\ge 2$. Notably, two recent contributions have focused on simulating the curve diffusion to achieve good mesh quality through harmonic maps \cite{Duan2023} and a specifically chosen tangential motion \cite{Deckelnick2024}. For a comprehensive introduction to PFEMs and other numerical techniques for curve diffusion, we refer to \cite{DDE2005,BGN20}.

Recently, structure-preserving schemes have attracted particular interest in designing efficient and accurate PFEMs. One of the seminal contributions to this field is the study conducted by Jiang and Li \cite{Jiang21}, who proposed the first fully discrete scheme for curve diffusion, demonstrating that it is both perimeter-decreasing and area-preserving. Subsequently, Bao and Zhao \cite{Bao-Zhao} developed a weakly implicit scheme for curve/surface diffusion by approximating the normal vector using the information from adjacent time steps. This methodology has been widely applied across various contexts, including axisymmetric geometric evolution equations \cite{Bao-Garcke-Nurnberg-Zhao}, surface diffusion for curve networks and surface clusters \cite{Bao-Garcke-Nurnberg-Zhao2023}, anisotropic surface diffusion \cite{Bao-Jiang-Li,Bao-Li2023}, and two-phase Navier-Stokes problems \cite{Garcke-Nurnberg-Zhao2023}. However, these schemes are limited to first-order temporal accuracy due to the use of the backward Euler time discretization, and extending them to incorporate high-order temporal discretization methods poses significant challenges
(see recent papers \cite{Jiang23, Jiang24} for further discussion).

Curve diffusion is an asymptotic limit of the  Cahn-Hilliard equation with degenerate mobility \cite{Cahn-Elliott-Cohen,Garcke-Cohen}, facilitating the investigation of moving interface problems via the phase field formulation. Notably, curve diffusion can be described as an $H^{-1}$-gradient flow of the area functional, whereas the Cahn-Hilliard equation represents an $H^{-1}$-gradient flow of the Ginzburg-Landau energy \cite{Garcke2013,Du-Feng,Taylor-Cahn}. In recent decades, considerable progress has been made in numerical methods for phase field models, notably including the invariant energy quadratization (IEQ) method \cite{Yang,Zhao-Wang-Yang}, the scalar auxiliary variable (SAV) method \cite{Shen2018,Shen2019,Huang2020}, the Lagrange multiplier approach \cite{Cheng2020}, and the convex splitting method \cite{Eyre}. These techniques have garnered significant attention due to their ability to facilitate the development of energy-stable schemes that achieve higher temporal accuracy. For a thorough examination of geometric interface evolution using phase field methods, we refer to \cite{Du-Feng}. Among these methods, the Lagrange multiplier approach has proven effective in tackling a variety of nonlinear partial differential equations and demonstrates compatibility with different time discretization strategies, all while maintaining essential physical structures \cite{Cheng2022A,Cheng2023,Badia,Krause-Voigt,Xia,BGN2019,Elliott-Stinner}. Furthermore, recent efforts have sought to adapt widely used temporal discretization methods for phase field models to address curve diffusion within the sharp-interface framework. For instance, Dai {\it et al.} \cite{Dai} utilized the convex splitting method to tackle strongly anisotropic curve diffusion, while Jiang {\it et al.} \cite{Jiang23,Jiang/pre,Jiang24} implemented the predictor-corrector or backward differentiation formulae for time discretization to develop efficient PFEMs that achieve high-order temporal accuracy and maintain good mesh quality.
It is important to note that applying these methods to the sharp interface model is not a trivial or straightforward task, as standard approaches may be ineffective due to mesh distortion issues \cite{Duan2023,Jiang23,Jiang24}.

The main objective of this paper is to develop several PFEMs based on the Lagrange multiplier approach for curve diffusion. These numerical schemes are designed to ensure perimeter reduction and area preservation at the fully discrete level, while simultaneously achieving high-order temporal accuracy and maintaining good mesh quality throughout the evolution. We describe the curve diffusion by a family of simple closed curves $\Gamma(t),\ t\in [0,T]$ which satisfies the following equation
\begin{equation}
	v=\p_{ss}\kappa,
\end{equation}
where $v$ represents the normal velocity, $s$ is the arc-length and $\kappa$ is the curvature. Denote $L(t):=L(\Gamma(t))$ and $A(t):=A(\Gamma(t))$  by the perimeter and area of the curve $\Gamma(t)$, respectively. It can be easily derived that \cite{DDE2005,Garcke2013}
\begin{align}
	&\frac{\d L}{\d t}=\int_{\Gamma(t)}\kappa  v \ \d s=\int_{\Gamma(t)}\kappa\,\p_{ss}\kappa \ \d s=-\int_{\Gamma(t)}|\p_s\kappa|^2\ \d s,\label{Evo of L:model}\\
	&\frac{\d A}{\d t}=\int_{\Gamma(t)}v\ \d s=\int_{\Gamma(t)}\p_{ss}\kappa\ \d s=0,\label{Evo of A:model}
\end{align}
thus the perimeter decreases and the area is conserved.

To introduce our new formulation, we first use the parametrization of curves $\bX(\rho,t):\I\times [0,T]\rightarrow \Gamma(t)$, $\I=\mathbb{R}/\mathbb{Z}$, and apply the formulation proposed by Barrett, Garcke and N\"urnberg \cite{BGN07A} (abbreviated as BGN), which incorporates implicit tangent motion and seeks $\bX$ and $\kappa$ such that
\begin{subequations}\label{BGNmodel}
  \begin{align}
  	\p_{t}\bX\cdot\bn &=\p_{ss}\kappa,\label{BGN1:model}\\
	 \kappa\bn &= -\p_{ss}\bX.\label{BGN2:model}
  \end{align}	
\end{subequations}
To design structure-preserving methods, inspired by \cite{BGN20,Cheng2020}, we introduce two scalar Lagrange multipliers $\lambda(t)$ and $\eta(t)$ to the first equation to enforce that \eqref{Evo of L:model} and \eqref{Evo of A:model} hold. The resulting formulation is given as follows: Find $\bX$, $\kappa$, $\lambda$, and $\eta$ such that
\begin{subequations}\label{Lag}
\begin{align}
	\p_t\bX\cdot \bn
		&=\p_{ss}\kappa+\lambda(t) \kappa+\eta(t)\label{Lag1} ,\\
		\kappa\bn &= -\p_{ss}\bX,\label{Lag2}\\
		\frac{\d L}{\d t}&=-\int_{\Gamma(t)}|\p_s\kappa|^2 \ \d s,\label{Lag3}\\
		\frac{\d A}{\d t}&=0.\label{Lag4}
\end{align}	
\end{subequations}	
Obviously, if $(\bX, \kappa)$ satisfies  \eqref{BGNmodel}, then $(\bX, \kappa, 0, 0)$ is a solution of \eqref{Lag}. Compared to the original BGN formulation \eqref{BGNmodel}, the newly proposed formulation \eqref{Lag} enjoys the following key advantages:
\begin{itemize}
\item [(1)] It can be readily demonstrated that, at the continuous level, the new formulation \eqref{Lag} reduces to the classical BGN formulation \eqref{BGNmodel}.
 Consequently, the discretization methods derived from the new formulation may retain the beneficial mesh quality attributes associated with the original BGN scheme \cite{BGN07A,BGN20};
\item [(2)] In comparison to prior structure-preserving schemes for curve diffusion \cite{Jiang21,Bao-Zhao}, our new formulation readily accommodates several high-order time discretization methods discussed in \cite{Jiang23,Jiang24,Jiang/pre},
including the Crank-Nicolson method and backward differentiation formulae method;
	
\item [(3)] It can be easily proved that the fully discrete schemes derived from the new formulation are both perimeter-decreasing and area-preserving during the evolution.
\end{itemize}

In practical simulations, extensive numerical results have demonstrated that our schemes can be implemented efficiently prior to reaching equilibrium. This implementation results in a reduction of the perimeter while successfully conserving area and achieving the desired temporal accuracy alongside good mesh quality. Once equilibrium is attained, the curve evolution can be continued by removing the Lagrange multiplier which corresponds to the perimeter-decreasing property, thereby only maintaining the area as constant. In addition, we present two alternative formulations that incorporate a single Lagrange multiplier to maintain either the perimeter-decreasing or area-preserving property. The resulting schemes exhibit high-order temporal accuracy and asymptotic mesh equidistribution during prolonged simulations.

The remainder of the paper is organized as follows. Section 2 delves into a thorough examination of the relationship between our proposed formulation and the classical BGN formulation, and presents the spatial semi-discretization. In Section 3, based on the new formulation, we present several time discretization methods for solving the curve diffusion, and prove the structure-preserving properties for the proposed fully discrete schemes. In Section 4, we explore some alternative approaches for designing high-order in time, structure-preserving PFEMs by using a single Lagrange multiplier which can maintain only one geometric structure. Section 5 presents extensive experiments to demonstrate the structure-preserving properties and the high-order temporal accuracy of the proposed schemes. Finally, we draw some conclusions in Section 6.

\section{Connection to the BGN formulation and spatial semi-discretization}

This section examines the relationship between our proposed formulation  \eqref{Lag} and the BGN formulation \eqref{BGNmodel}, and presents the spatial semi-discretization associated with \eqref{Lag}.

\subsection{Connection to the BGN formulation} By employing a parametrization $\bX(\cdot,t):\I\rightarrow \Gamma(t)\subset\R^2$ alongside the arc-length $s(\rho,t)$, we can identify the function over the closed curve $\Gamma(t)$ and the periodic interval $\I=[0,1]$. For the curve $\Gamma=\bX(\I)$, we introduce the following function space
\[
	L^2(\Gamma):=\Big\{u:\Gamma\rightarrow \R \,|\, \int_{\Gamma}|u(s)|^2\d s=\int_{\I}\big|u\circ \bX (\rho)\big|^2\ |\p_\rho\bX|\ \d \rho<+\infty \Big\},
\]
with the inner product given by
\begin{equation}\label{L^2 inner}
	\l(u,v\r)_{\Gamma}:=\int_{\Gamma}\, u(s)\, v(s)\ \d s=\int_{\I} \l(u\circ\bX(\rho)\r)\, \l(v\circ\bX(\rho)\r)\, |\p_\rho \bX|\ \d\rho.
\end{equation}

It is well-known that a circle \cite{Elliott-Garcke}, characterized by constant curvature, represents one equilibrium state of curve diffusion. In the following proposition, we demonstrate that the two Lagrange multipliers are both zero and our new formulation \eqref{Lag} will reduce to the BGN formulation \eqref{BGNmodel} before the evolving curve reaches its equilibrium state, i.e., a circle.

\smallskip

\begin{proposition}\label{Prop:cont}
 Assuming that $(\bX,\kappa,\lambda,\eta)$ constitutes a solution to \eqref{Lag} and $\Gamma(t)=\bX(\I,t)$ represents a closed $C^2$-curve, it follows that $\lambda=\eta=0$ if $\kappa(\cdot, t)$ is not constant with respect to the spatial variable.
 \end{proposition}

\begin{proof}
	Multiplying $\kappa$ and $\p_t\bX$ on both sides of \eqref{Lag1} and \eqref{Lag2}, respectively, integrating over the curve $\Gamma(t)$, noting $\p_t\bX\cdot\bn=v$ is the normal velocity, we obtain
	\begin{align*}
		\lambda(t)\l(\kappa,\kappa\r)_{\Gamma(t)}+\eta(t)\l(1,\kappa\r)_{\Gamma(t)}
	&= 	\l(\p_t\bX,\kappa\, \bn \r)_{\Gamma(t)}+\l( \p_s \kappa,\p_s\kappa \r)_{\Gamma(t)}\\
	&=\l(v,\kappa\r)_{\Gamma(t)}+\l( \p_s \kappa,\p_s\kappa \r)_{\Gamma(t)}\\
	&=\frac{\d L}{\d t}+\l( \p_s \kappa,\p_s\kappa \r)_{\Gamma(t)}=0,
	\end{align*}
where we used integration by parts, \eqref{Lag3} and \eqref{Evo of L:model}.
Similarly, 	multiplying $1$ and $\p_t\bX$ on the both sides of \eqref{Lag1} and \eqref{Lag2}, respectively, integrating over the curve $\Gamma(t)$, utilizing integration by parts, \eqref{Lag4} and \eqref{Evo of A:model}, we get
\[
	\lambda(t)\l( \kappa,1\r)_{\Gamma(t)}+\eta(t) \l( 1,1 \r)_{\Gamma(t)}
	=\l(v,1\r)_{\Gamma(t)}+\l( \p_s \kappa,\p_s 1 \r)_{\Gamma(t)}=\frac{\d A}{\d t}=0.
\]
Therefore, $\lambda(t)$ and $\eta(t)$ satisfy a time dependent algebraic system
\begin{equation*}
	\bM(t)\begin{bmatrix}
		\lambda(t)\\
		\eta(t)
	\end{bmatrix}=0,\qquad \bM(t)=\begin{bmatrix}
		\l(\kappa,\kappa\r)_{\Gamma(t)} & \l(\kappa,1\r)_{\Gamma(t)}\\
		\l(\kappa,1\r)_{\Gamma(t)} & \l(1,1\r)_{\Gamma(t)}
	\end{bmatrix}.
\end{equation*}
By the Cauchy-Schwarz inequality, we have
\[
	\det\l(\bM(t)\r)
	=\l(\kappa,\kappa\r)_{\Gamma(t)}\l(1,1\r)_{\Gamma(t)}-
\l(\kappa,1\r)_{\Gamma(t)}^2\ge 0,
\]
 and the equality holds if and only if $\kappa(\cdot, t)$ is constant (with respect to spatial variable). Consequently, the matrix $\bM(t)$ is invertible if $\kappa(\cdot, t)$ is not constant, which implies that $\lambda(t)=\eta(t)=0$, thus concluding the proof. \end{proof}

\subsection{Spatial semi-discrete scheme}

To introduce the spatial discretization, let $\mathbb{I}=[0,1]= \bigcup_{j=1}^N I_j$, $N\ge 3$, where $I_j=[\rho_{j-1},\rho_j]$ with $0=\rho_0<\rho_1<\ldots<\rho_N=1$. Denote $h=\max\limits_{1\le j\le N}
|\rho_j-\rho_{j-1}|$ by the maximal length of the grid. Define the linear finite element space as
\[
V^h:=\{u\in C(\mathbb{I}): u|_{I_j} \,\,\, \mathrm{is\,\,\,linear,\,\,\,} \forall j=1,2,\ldots,N;\quad u(\rho_0)=u(\rho_N) \}\subseteq H^1(\mathbb{I}).
\]
The inner product \eqref{L^2 inner} is applicable to the polygonal curve $\Gamma^h=\bX^h(\I)$, $\bX^h\in [V^h]^2$. In practice, we approximate it using the mass lumped inner product $(\cdot,\cdot)_{\Gamma^h}^h$ obtained by the composite trapezoidal rule, i.e.,
\begin{equation}\label{massp}
(u,v)_{\Gamma^h}^h:=\frac{1}{2}\sum_{j=1}^N|\mathbf{h}_j|\l[(u\cdot v)(\rho_j^-)+(u\cdot v)(\rho_{j-1}^+) \r], \quad \mathbf{h}_j=\bX^h(\rho_j)-\bX^h(\rho_{j-1}),
\end{equation}
where $\bX^h(\rho_j)$ is the vertex of the polygon $\Gamma^h$, and  $u, v$ are two scalar/vector piecewise continuous functions with possible jumps at the nodes $\{\rho_j\}_{j=0}^N$,
and $u(\rho_j^{\pm})=\lim\limits_{\rho\rightarrow \rho_j^{\pm}}u(\rho)$.
It can be easily seen that the mass lumped inner product coincides with the inner product \eqref{L^2 inner} if $u\cdot v$ is piecewise linear.

The semi-discrete scheme reads as follows: Given initial polygon $\Gamma^h(0)=\bX^h(\I,0)$, $\bX^h(\cdot, 0)\in [V^h]^2$, find $(\bX^h(\cdot,t),\kappa^h(\cdot,t), \lambda^h(t),\eta^h(t))\in [V^h]^2\times V^h\times \R\times \R$ such that for any $(\varphi^h, \bm{\omega}^h)\in V^h\times [V^h]^2$, it holds
\begin{subequations}\label{Lagsemi}
	\begin{align}
		\l(\p_t\bX^h,\varphi^h \bn^h \r)^h_{\Gamma^h}+\l( \p_s \kappa^h,\p_s\varphi^h \r)_{\Gamma^h}&-\lambda^h\l( \kappa^h,\varphi^h \r)_{\Gamma^h}^h-\eta^h\l( 1,\varphi^h \r)_{\Gamma^h}^h =0,\label{Lag1:semi} \\
 \l(\kappa^h\bn^h,\bm{\omega}^h\r)_{\Gamma^h}^h&-\l(\p_s \bX^h,\p_s\bm{\omega}^h\r)_{\Gamma^h}=0,\label{Lag2:semi}\\
\frac{\d L^h}{\d t}&+\l( \p_s\kappa^h,\p_s\kappa^h\r)_{\Gamma^h}=0,\label{Lag3:semi} \\
	\frac{\d A^h}{\d t}&=0,\label{Lag4:semi}
	\end{align}
\end{subequations}
where $\Gamma^h=\Gamma^h(t)=\bX^h(\I,t)$, the normal vector $\bn^h$ and the tangential gradient $\p_s$ are both piecewisely defined via
\begin{equation}\label{ndef}
\bn^h|_{I_j}=-\frac{\mathbf{h}_j^\perp}{|\mathbf{h}_j|}, \quad \p_s f|_{I_j}=\l.\frac{\p_\rho f}{|\p_\rho \mathbf{X}^h|}\r|_{I_j}=\frac{(\rho_j-\rho_{j-1})\p_\rho f|_{I_j}}{|\mathbf{h}_j|},\quad j=1,\ldots, N,
\end{equation}
where $(\cdot)^{\perp}$ is the clockwise rotation by $\frac{\pi}{2}$.
Moreover, the  perimeter $L^h(t)=L(\bX^h)$ and area $A^h(t)=A(\bX^h)$ are given by
\begin{align}
	L(\bX)&:=\int_{\Gamma(\bX)}1\, \d s=\int_{\I}|\p_\rho \bX| \ \d \rho ,\label{Def of L:cont} \\
   A(\bX)&:=\frac{1}{2}\int_{\Gamma(\bX)}\bX\cdot \bn\, \d s=-\frac{1}{2}\int_{\I}\bX\cdot(\p_{\rho}\bX)^\perp\, \d\rho.\label{Def of A:cont}
\end{align}
Concerning the properties of $L$ and $A$, we have the following lemma.

\smallskip
\begin{lemma}\label{Lem:Evo/weak}
	Let $\Gamma(t)=\bX(\I,t)$ represent a family of continuous, piecewise $C^1$ closed curves, with possible jumps at the nodes $\{\rho_j\}_{j=0}^N$ with $0=\rho_0<\rho_1<\ldots<\rho_N=1$ and $|\p_\rho \bX|\neq 0$. Moreover, assume that both $\bX$ and $\p_\rho \bX$ are differentiable with respect to time. Then the perimeter and area can be expressed as follows:
	\begin{align}
		\frac{\d L}{\d t}&=\l(\p_s\bX,\p_s\p_t\bX\r)_{\Gamma(t)},\label{Evo of L:cont}\\
		\frac{\d A}{\d t}&=\l(\p_t\bX,\bn\r)_{\Gamma(t)},\label{Evo of A:cont}
	\end{align}
	where $\bn$ and $\p_s$ are piecewisely defined normal and tangential gradient.
\end{lemma}
\begin{proof}
Differentiating $L$ and $A$ gives
\begin{align*}
	\frac{\d L}{\d t}
 	&=\frac{\d}{\d t}\int_{\I}|\p_\rho\bX|\ \d\rho=\sum_{j=1}^N\int_{\rho_{j-1}}^{\rho_j} \frac{\p_\rho \bX\cdot \p_\rho \p_t\bX}{|\p_\rho\bX|} \ \d\rho=\l(\p_s\bX,\p_s\p_t\bX\r)_{\Gamma(t)},\\
 	\frac{\d A}{\d t}
 	&=-\frac{1}{2}\frac{\d}{\d t}\int_{\I}\bX\cdot(\p_\rho\bX)^{\perp}\ \d \rho\\
 	&=-\frac{1}{2}\sum_{j=1}^N\int_{\rho_{j-1}}^{\rho_j}\p_t\bX\cdot(\p_\rho\bX)^{\perp}+\bX \cdot (\p_t\p_{\rho}\bX)^\perp \ \d \rho \\
 	&=\frac{1}{2}\l(\p_t\bX,\bn\r)_{\Gamma(t)}+\frac{1}{2}\sum_{j=1}^N\int_{\rho_{j-1}}^{\rho_j}\p_{\rho}\bX\cdot (\p_t\bX)^\perp\ \d\rho-\frac{1}{2}\sum_{j=1}^N\l[\bX\cdot (\p_t\bX)^\perp \r]\big|_{\rho_{j-1}}^{\rho_j} \\
 &=\frac{1}{2}\l(\p_t\bX,\bn\r)_{\Gamma(t)}-\frac{1}{2}\sum_{j=1}^N\int_{\rho_{j-1}}^{\rho_j}\p_t\bX\cdot(\p_\rho\bX)^{\perp}\ \d\rho\\
 	&=\l(\p_t\bX,\bn\r)_{\Gamma(t)},
\end{align*}
where we have used the continuity and the property $\mathbf{a}\cdot \mathbf{b}^\perp=-\mathbf{a}^\perp\cdot \mathbf{b}$ for two vectors $\mathbf{a},\mathbf{b}\in \R^2$, and the proof is completed.
\end{proof}

Applying Lemma \ref{Lem:Evo/weak} to $L^h$ and $A^h$, and noticing the mass lumped inner product \eqref{massp} is equal to the inner product \eqref{L^2 inner} for piecewise linear functions, we can get
\begin{align}
	\frac{\d L^h}{\d t}&=\l(\p_s\bX^h,\p_s\p_t\bX^h\r)_{\Gamma^h},\label{Evo of L:semi}\\
	\frac{\d A^h}{\d t}&=\l(\p_t\bX^h,\bn^h\r)_{\Gamma^h}=\l(\p_t\bX^h,\bn^h\r)^h_{\Gamma^h}.\label{Evo of A:semi}
\end{align}
Combining the evolution formulae \eqref{Evo of L:semi} and \eqref{Evo of A:semi}, we can generalize  Proposition \ref{Prop:cont} to the semi-discrete framework.

\smallskip

\begin{proposition}\label{Prop:semi}
 	Suppose $(\bX^h(\cdot,t),\kappa^h(\cdot,t), \lambda(t),\eta(t))\in [V^h]^2\times V^h\times \R\times \R$ is the solution of \eqref{Lagsemi}, then $\lambda^h(t)=\eta^h(t)=0$ if $\kappa^h(\cdot,t)$ is not constant.
 \end{proposition}
\begin{proof}
  Taking $\varphi=\kappa^h$ in \eqref{Lag1:semi} and $\bm{\omega}^h=\p_t\bX^h$ in \eqref{Lag2:semi}, combining with \eqref{Lag3:semi} and \eqref{Evo of L:semi}, we obtain
\begin{align*}
	\lambda^h\l(\kappa^h,\kappa^h\r)^h_{\Gamma^h}+\eta^h\l(1,\kappa^h\r)^h_{\Gamma^h}
	&= 	\l(\p_t\bX^h,\kappa^h\, \bn^h \r)^h_{\Gamma^h}+\l( \p_s \kappa^h,\p_s\kappa^h \r)_{\Gamma^h}\\
	&=\l(\p_s\bX^h,\p_s\p_t\bX^h\r)_{\Gamma^h}+\l( \p_s \kappa^h,\p_s\kappa^h \r)_{\Gamma^h}\\
	&=\frac{\d L^h}{\d t}+\l( \p_s \kappa^h,\p_s\kappa^h\r)_{\Gamma^h}=0.
\end{align*}
On the other hand, by taking $\varphi^h=1$ in \eqref{Lag1:semi} and $\bm{\omega}^h=\p_t\bX^h$ in \eqref{Lag2:semi}, combining with \eqref{Lag4:semi} and \eqref{Evo of A:semi}, we get
\[
	\lambda^h\l( \kappa^h,1\r)^h_{\Gamma^h}+\eta^h \l( 1,1 \r)^h_{\Gamma^h}
	=\l(\p_t\bX^h, \bn^h \r)^h_{\Gamma^h}+\l( \p_s \kappa^h,\p_s 1 \r)_{\Gamma^h}=\frac{\d A^h}{\d t}=0.
\]
 We complete the proof by invoking the Cauchy-Schwarz inequality for the inner product $(\cdot,\cdot)^h_{\Gamma^h}$ and a similar argument in Proposition \ref{Prop:cont}. \end{proof}

\smallskip

\begin{remark}\label{Rem:semi}
	Similar to the continuous case, if $\kappa^h(\cdot, t)$ is not constant, then $\lambda^h(t)=\eta^h(t)=0$ and the semi-discrete scheme \eqref{Lagsemi} collapses to the classical BGN's semi-discrete scheme for curve diffusion \cite{BGN07A}:
	\begin{subequations}\label{BGNsemi}
	\begin{align}
		\l(\p_t\bX^h,\varphi^h \bn^h \r)^h_{\Gamma^h}+\l( \p_s \kappa^h,\p_s\varphi^h \r)_{\Gamma^h}&=0,\label{BGN1:semi} \\
 \l(\kappa^h\bn^h,\bm{\omega}^h\r)_{\Gamma^h}^h-\l(\p_s \bX^h,\p_s\bm{\omega}^h\r)_{\Gamma^h}&=0.\label{BGN2:semi}
	\end{align}
\end{subequations}
This formulation is well-known for its equidistribution property  \cite{BGN07A,BGN20,Zhao2021}. Moreover, it is noteworthy that the first-order structure-preserving scheme presented in \cite{Bao-Zhao} was established based on \eqref{BGNsemi}.
\end{remark}

\section{Time stepping discretization}\label{sec:SP}

As demonstrated in Propositions \ref{Prop:cont} and \ref{Prop:semi} (see also Remark \ref{Rem:semi}), the formulation \eqref{Lag} aligns with the classical BGN formulation \eqref{BGNmodel} until the equilibrium is reached. However, it is distinctive in its capacity to facilitate the development of higher-order accuracy, structure-preserving schemes through various time discretization methods. To incorporate time-stepping techniques for numerical solutions that involve implicit tangential motion, we employ the prediction methods detailed in \cite{Jiang23,Jiang/pre,Jiang24}. For the purpose of full discretization, we assume that $\bX^m\in [V^h]^2$ and $\Gamma^m=\bX^m(\I)$ are approximations of $\bX(\cdot,t_m)$ and $\Gamma(t_m)$, respectively, for $m=0,1,2,\ldots$, where $t_m:=m\tau$ and $\tau$ denotes a uniform time increment.

\subsection{Backward Euler time discretization}

As the most straightforward time discretization technique, we utilize the backward Euler discretization to the semi-discrete scheme \eqref{Lagsemi}, which reads as (denoted as the \textbf{SP-Euler scheme}): For $m\ge 0$, seek $(\mathbf{X}^{m+1},\kappa^{m+1}, \lambda^{m+1},\eta^{m+1})\in [V^h]^2\times V^h\times \R\times \R$ such that
\begin{subequations}\label{Lageuler}
	\begin{align}\label{Lag1:E}
		\begin{split}
			\l(\frac{\bX^{m+1}-\bX^m}{\tau},\varphi^h \bn^{m} \r)^h_{\Gamma^{m}}&+\l( \p_s \kappa^{m+1},\p_s\varphi^h \r)_{\Gamma^{m}}-\lambda^{m+1}\l( \kappa^{m+1},\varphi^h \r)_{\Gamma^{m}}^h \\
	&-\eta^{m+1}(1,\varphi^h)_{\Gamma^{m}}^h =0,\quad \forall \varphi^h\in V^h,\\
		\end{split}
	\end{align}
	\vspace{-15pt}
\begin{align}
	\l(\kappa^{m+1}\bn^{m},\bm{\omega}^h\r)_{\Gamma^{m}}^h&-\l(\p_s \bX^{m+1},\p_s\bm{\omega}^h\r)_{\Gamma^{m}}=0,\quad \forall \bm{\omega}^h\in [V^h]^2,\label{Lag2:E}\\
\frac{L^{m+1}-L^m}{\tau}&+\l( \p_s\kappa^{m+1},\p_s\kappa^{m+1}\r)_{\Gamma^{m}}=0,\label{Lag3:E} \\
	A^{m+1}-A^m&=0.\label{Lag4:E}
\end{align}
\end{subequations}
Here the normal vector $\bn^{m}$ and the partial derivative $\p_s$ are both defined for $\Gamma^m=\bX^{m}(\I)$ in a piecewise manner as outlined in \eqref{ndef}. Additionally,  $L^m=L(\bX^m)$ and  $A^m=A(\bX^m)$ denote the perimeter and area of the polygon $\Gamma^{m}$, respectively.
	
The following structure-preserving property is a direct consequence of \eqref{Lag3:E} and \eqref{Lag4:E}.

\smallskip
	
\begin{theorem}\label{Thm:E}
	Let $(\bX^{m+1},\kappa^{m+1}, \lambda^{m+1},\eta^{m+1})\in [V^h]^2\times V^h\times \R\times \R$ be the solution to the Euler scheme \eqref{Lageuler}. It then follows that
	\begin{equation*}
		L^{m+1}\le L^{m}\le \ldots\le L^0 ,\qquad A^{m+1}=A^m=\ldots\equiv A^0,\qquad m\ge 0.
	\end{equation*}
\end{theorem}	

\smallskip

\begin{remark}
Compared to the first-order structure-preserving scheme proposed by Bao and Zhao \cite{Bao-Zhao}, our new scheme \eqref{Lageuler} enforces both  structure-preserving laws explicitly, with the primary focus on nonlinearity occurring in the final two equations \eqref{Lag3:E} and \eqref{Lag4:E}.
\end{remark}

\subsection{Crank-Nicolson time discretization}

Using the methodology outlined in \cite{Jiang/pre},  we subsequently apply the Crank-Nicolson discretization method to \eqref{Lagsemi}, thereby deriving the following scheme (denoted as the \textbf{SP-CN scheme}): for $m\ge 0$, find $(\mathbf{X}^{m+1},\kappa^{m+1}, \lambda^{m+1},\eta^{m+1})\in [V^h]^2\times V^h\times\R\times \R$ such that for all $(\varphi^h, \bm{\omega}^h)\in V^h\times [V^h]^2$, it holds
\begin{subequations}\label{LagCN}
\begin{align}\label{Lag1:CN}
  \begin{split}
	 \Big(\frac{\bX^{m+1}-\bX^m}{\tau},\varphi^h \widetilde{\bn}^{m+1/2} \Big)^h_{\widetilde{\Gamma}^{m+1/2}}&+\l( \p_s \kappa^{m+1/2},\p_s\varphi^h \r)_{\widetilde{\Gamma}^{m+1/2}} \\
	 -\lambda^{m+1/2}\l( \kappa^{m+1/2},\varphi^h \r)_{\widetilde{\Gamma}^{m+1/2}}^h &-\eta^{m+1/2}(1,\varphi^h)_{\widetilde{\Gamma}^{m+1/2}}^h =0,\\
  \end{split}
  \end{align}	
  \vspace{-15pt}
  \begin{align}
 \l(\kappa^{m+1/2}\widetilde{\bn}^{m+1/2},\bm{\omega}^h\r)_{\widetilde{\Gamma}^{m+1/2}}^h&-\l(\p_s \bX^{m+1/2},\p_s\bm{\omega}^h\r)_{\widetilde{\Gamma}^{m+1/2}}=0,\label{Lag2:CN}\\
\frac{L^{m+1}-L^m}{\tau}&+\l( \p_s\kappa^{m+1/2},\p_s\kappa^{m+1/2}\r)_{\widetilde{\Gamma}^{m+1/2}}=0,\label{Lag3:CN} \\
	A^{m+1}-A^m&=0,\label{Lag4:CN}
\end{align}
\end{subequations}
where
\[
\kappa^{m+1/2}=\frac{\kappa^{m+1}+\kappa^m}{2},\quad \lambda^{m+1/2}=\frac{\lambda^{m+1}+\lambda^{m}}{2},\quad  \eta^{m+1/2}=\frac{\eta^{m+1}+\eta^{m}}{2},
\]
and $\widetilde{\Gamma}^{m+1/2}=\widetilde{\bX}^{m+1/2}(\I)$ with  $\widetilde{\bX}^{m+1/2}$ being the solution of the SP-Euler scheme \eqref{Lageuler} with half time step $\tau/2$, the normal vector $\widetilde{\bn}^{m+1/2}$ and the partial derivative $\p_s$ are both piecewisely defined over $\widetilde{\Gamma}^{m+1/2}$.
The scheme \eqref{LagCN} is also structure-preserving in view of \eqref{Lag3:CN} and \eqref{Lag4:CN}.

\smallskip

\begin{theorem}\label{Thm:CN}
The solution of the SP-CN scheme \eqref{LagCN} satisfies
	\[
		L^{m+1}\le L^{m}\le \ldots\le L^0 ,\qquad A^{m+1}=A^m=\ldots\equiv A^0,\qquad m\ge 0.
	\]
\end{theorem}

\subsection{Backward differentiation formulae for time discretization}

Using the idea presented in \cite{Jiang24}, we apply the backward differentiation formulae (BDF) for time discretization to the semi-discrete scheme \eqref{Lagsemi}. For the sake of simplicity, we focus on the second-order temporal approximation (referred to as the \textbf{SP-BDF2 scheme}): given $\Gamma^0$ and $\Gamma^{1}$, which serve as approximations of $\Gamma(0)$ and $\Gamma(\tau)$, respectively, we seek $(\mathbf{X}^{m+1},\kappa^{m+1}, \lambda^{m+1},\eta^{m+1})\in [V^h]^2\times V^h\times \R\times \R$ for $m\ge 1$, such that
\begin{subequations}\label{LagBDFk}
\begin{align}\label{Lag1:BDFk}
   \begin{split}
   	   \Big(\frac{\frac{3}{2}\bX^{m+1}-2\bX^m+\frac{1}{2}\bX^{m-1}}
   {\tau},\varphi^h \widetilde{\bn}^{m+1} \Big)^h_{\widetilde{\Gamma}^{m+1}}&+\l( \p_s \kappa^{m+1},\p_s\varphi^h \r)_{\widetilde{\Gamma}^{m+1}} \\
-\lambda^{m+1}\l( \kappa^{m+1},\varphi^h \r)_{\widetilde{\Gamma}^{m+1}}^h-\eta^{m+1}&(1,\varphi^h)_{\widetilde{\Gamma}^{m+1}}^h =0,\\
   \end{split}
\end{align}
\vspace{-13pt}
\begin{align}
 \l(\kappa^{m+1}\widetilde{\bn}^{m+1},\bm{\omega}^h\r)_{\widetilde{\Gamma}^{m+1}}^h-\l(\p_s \bX^{m+1},\p_s\bm{\omega}^h\r)_{\widetilde{\Gamma}^{m+1}}&=0,\label{Lag2:BDFk}\\
\frac{\frac{3}{2}L^{m+1}-2{L}^m+\frac{1}{2}L^{m-1}}{\tau}+\l( \p_s\kappa^{m+1},\p_s\kappa^{m+1}\r)_{\widetilde{\Gamma}^{m+1}}&=0,\label{Lag3:BDFk} \\
	\frac{3}{2}A^{m+1}-2A^m+\frac{1}{2}A^{m-1}&=0,\label{Lag4:BDFk}
\end{align}
\end{subequations}
where  $\widetilde{\Gamma}^{m+1}$, described by $\widetilde{\bX}^{m+1}\in [V^h]^2$, is predicted by the Euler scheme \eqref{Lageuler} with a full time step $\tau$. The normal vector $\widetilde{\bn}^{m+1}$ and the tangential gradient $\p_s$ are both piecewisely defined over $\widetilde{\Gamma}^{m+1}$. Moreover, the approximation $\Gamma^1$ for the first time step is obtained using the Euler scheme \eqref{Lageuler} with one time step $\tau$. Higher-order BDF$k(k\ge2)$ algorithms can be proposed similarly. For more details, we refer to \cite{Jiang24}.

\smallskip

\begin{theorem}\label{Thm:BDF}
 Let $(\bX^{m+1},\kappa^{m+1}, \lambda^{m+1},\eta^{m+1})\in [V^h]^2\times V^h\times\R\times \R$ be the solution of the  SP-BDF$2$ scheme \eqref{LagBDFk}, then it holds
	\[
		L^{m+1}\le L^{m}\le \ldots\le L^0 ,\qquad A^{m+1}=A^m=\ldots\equiv A^0,\qquad m\ge 0.
	\]
\end{theorem}	
\begin{proof}
We prove the following by induction:
	\[
		L^{m+1}\le L^m,\qquad A^{m+1}=A^m.
	\]
Recalling $\bX^1$ is obtained by the SP-Euler scheme \eqref{Lageuler}, by Theorem \ref{Thm:E}, we have
\[
L^{1}\le L^0, \qquad A^1=A^0.
\]
For $m\ge 1$, the equations \eqref{Lag3:BDFk} and \eqref{Lag4:BDFk} imply that
	\begin{equation*}
		\frac{3}{2}L^{m+1}-2L^{m}+\frac{1}{2}L^{m-1}\le 0,\qquad 	\frac{3}{2}A^{m+1}-2A^{m}+\frac{1}{2}A^{m-1}= 0.
	\end{equation*}
	Therefore, by induction, we have
	\begin{align*}
		\frac{3}{2}\l(L^{m+1}-L^m\r)\le \frac{1}{2}\l(L^m-L^{m-1}\r)&\le 0,\\
		\frac{3}{2}\l(A^{m+1}-A^m\r)= \frac{1}{2}\l(A^m-A^{m-1}\r)&= 0,
	\end{align*}
	and the proof is completed. \end{proof}

\smallskip

\begin{remark}
It is feasible to derive the BDF\,$k$ schemes for $3\le k\le 6$ by using a similar approach to that described in \cite{Jiang24}. Moreover, by applying the arguments presented in Theorem \ref{Thm:BDF:AP}, it is possible to demonstrate  the area-preserving property.  However, proving the perimeter-decreasing property for these schemes appears challenging. The difficulty arises from the fact that the equation \eqref{Lag3:BDFk} associated with the BDF\,$k$ schemes (\,$3\le k\le 6$)
 does not guarantee the same unconditional energy stability as the BDF\,$2$ scheme.
For instance, \eqref{Lag3:BDFk} in BDF\,$3$ scheme gives
	\begin{align}
		0&\ge \frac{11}{6}L^{m+1}-\frac{18}{6}L^m+\frac{9}{6}L^{m-1}-\frac{2}{6}L^{m-2}\nonumber\\
     &=\frac{11}{6}(L^{m+1}-(1+a)L^m+aL^{m-1})-b(L^{m}-(1+a)L^{m-1}
     +aL^{m-2}).\label{decomp}
	\end{align}
If it is possible to find $a ,b>0$ such that the above equality holds, then by induction, we have
\begin{align*}
L^{m+1}-(1+a)L^m+aL^{m-1}&\le \frac{6b}{11}(L^{m}-(1+a)L^{m-1}
     +aL^{m-2})\\
     &\le \l(6b/11\r)^{m-1}(L^{2}-(1+a)L^{1}
     +aL^{0})\le 0,\end{align*}
if the initial step values satisfy $L^{2}-(1+a)L^{1}
     +aL^{0}\le 0$. It follows that
     \[L^{m+1}-L^m\le a(L^m-L^{m-1})\le a^{m}(L^1-L^0)\le 0,\]
if $L^1\le L^0$. Now it remains to seek positive $a$, $b$ such that \eqref{decomp} holds. Direct calculation yields
\[11a+6b=7,\quad ab=1/3,\]
unfortunately, this system does not admit real solutions for $a$ and $b$.
To our knowledge, there is currently no rigorous proof of energy stability for the BDF\,$k$ ($3\le k\le 6$) schemes when applied to the phase field model \cite{Shen2019}.
\end{remark}

\smallskip

\begin{remark}\label{Change third equation}
We stress that the discretization for \eqref{Lag3:semi}, e.g.,  \eqref{Lag3:BDFk}, is also crucial to maintain the desired temporal accuracy. For instance, consider replacing \eqref{Lag3:BDFk} by
\begin{equation}\label{Lag3:BDFk:change}
\frac{L^{m+1}-{L}^m}{\tau}+\l( \p_s\kappa^{m+1},\p_s\kappa^{m+1}\r)_{\widetilde{\Gamma}^{m+1}}=0.
\end{equation}
This adjustment would simplify the proof of the perimeter-decreasing property, however, it does not guarantee that all equations are discretized with consistent second-order temporal accuracy. Indeed, numerical experiments (cf. Figure \ref{Fig:EOC}) indicate that the convergence order is less than $2$ for the scheme \eqref{LagBDFk} with \eqref{Lag3:BDFk} replaced by \eqref{Lag3:BDFk:change}.
\end{remark}

\subsection{Newton's iteration method and modification for long-time evolution}

In this subsection we discuss the implementation of the nonlinear systems for the SP-Euler scheme \eqref{Lageuler}, the SP-CN scheme \eqref{LagCN} and the SP-BDF2 scheme \eqref{LagBDFk} by Newton's iteration. The implementation of Newton's iteration requires the variation of the perimeter and area functional. Indeed, let $\Gamma$ be a closed polygon  parametrized by $\bX\in [V^h]^2$. We consider a family of closed polygonal curves parametrized by $\bY(\varepsilon)=\bX+\varepsilon\bX^{\delta}$,  $\bX^{\delta}\in [V^h]^2$, then Lemma \ref{Lem:Evo/weak} yields that
\begin{align}
	\frac{\delta L(\bX)}{\delta \bX}(\bX^\delta)
	&=\l.\frac{\d L(\bY(\varepsilon))}{\d \varepsilon}\r|_{\varepsilon=0}=\l(\p_s\bX, \p_s\bX^\delta\r)_{\Gamma},\label{First variation of L}\\
	\frac{\delta A(\bX)}{\delta \bX}(\bX^\delta)
	&=\l.\frac{\d A(\bY(\varepsilon))}{\d \varepsilon}\r|_{\varepsilon=0}=\l(\bX^\delta,\bn \r)_{\Gamma}= \l(\bX^\delta,\bn \r)^h_{\Gamma},\label{First variation of A}
\end{align}
where we utilized the fact $\p_\varepsilon\bY=\bX^{\delta}$.

\smallskip

We take the SP-Euler scheme \eqref{Lageuler} as a representative example, and the other two systems can be solved in a similar manner. By the area-preserving property presented in Theorem \ref{Thm:E}, we can replace \eqref{Lag4:E} by a simpler one
\[
	A^{m+1}-A^0=0.
\]
Thus we get Newton's iteration as follows: For given $(\bX^{m},\kappa^{m}, \lambda^{m}, \eta^{m})\in [V^h]^2\times V^h\times \R \times \R$, we establish the initial predicted solution for the subsequent step as
\[
	\bX^{m+1,0}=\bX^{m},\quad \kappa^{m+1,0}=\kappa^{m},\quad
	\lambda^{m+1,0}=\lambda^{m},\quad \eta^{m+1,0}=\eta^{m}.
\]
In $i$-th Newton's iteration step, find the Newton's direction $(\bX^{\delta},\kappa^{\delta}, \lambda^{\delta},\eta^{\delta})\in [V^h]^2\times V^h\times\R\times \R$ such that for all $(\varphi_h, \bm{\omega}^h)\in V^h\times [V^h]^2$, it holds
\begin{subequations}\label{LagNewton}
	\begin{align}\label{Lag1:Newton}
		\begin{split}
			&\frac{1}{\tau}\l(\bX^{\delta},\varphi^h \bn^{m} \r)^h_{\Gamma^{m}}+\l( \p_s\kappa^{\delta},\p_s\varphi^h \r)_{\Gamma^{m}}^h -\lambda^\delta\l( \kappa^{m+1,i},\varphi^h \r)_{\Gamma^{m}}^h-\lambda^{m+1,i}\l(\kappa^{\delta},\varphi^h \r)_{\Gamma^{m}}^h \\
		&\qquad-\eta^\delta (1,\varphi^h)_{\Gamma^{m}}^h =-\l(\frac{\bX^{m+1,i}-\bX^{m} }{\tau},\varphi^h \bn^{m}\r)^h_{\Gamma^{m}} -\l( \p_s \kappa^{m+1,i},\p_s \varphi^h \r)_{\Gamma^{m}}^h \\
		&\qquad\qquad\qquad\qquad\quad\,\,\, +\lambda^{m+1,i}\l(\kappa^{m+1,i},\varphi^h \r)_{\Gamma^{m}}^h+\eta^{m+1,i}(1,\varphi^h)_{\Gamma^{m}}^h,\\
		\end{split}
	\end{align}
	\vspace{-7pt}
	\begin{align}
		\begin{split}\label{Lag2:Newton}
			&\l(\kappa^{\delta }\bn^{m},\bm{\omega}^h\r)_{\Gamma^{m}}^h-\l(\p_s \mathbf{X}^{\delta},\p_s\bm{\omega}^h\r)_{\Gamma^{m}}\\
	  &\qquad\qquad =-\l(\kappa^{m+1,i}\mathbf{n}^{m},\bm{\omega}^h\r)_{\Gamma^{m}}^h+\l(\p_s \mathbf{X}^{m+1,i},\p_s\bm{\omega}^h\r)_{\Gamma^{m}}, \\
		\end{split}
	\end{align}
	\vspace{-7pt}
	\begin{align}\label{Lag3:Newton}
	  \begin{split}
		&\frac{1}{\tau} \l(\p_s\bX^{m+1,i}, \p_s \bX^\delta\r)_{\Gamma^{m+1,i}}+2\l(\p_s \kappa^{\delta},\p_s\kappa^{m+1,i}\r)_{\Gamma^{m}} \\
	 &\qquad\qquad =-\frac{L^{m+1,i}-L^{m}}{\tau}-\l(\p_s \kappa^{m+1,i},\p_s\kappa^{m+1,i}\r)_{\Gamma^{m}}^h,
	  \end{split}	
	\end{align}
	\vspace{-7pt}
	\begin{align}\label{Lag4:Newton}
		&\l(\bX^\delta ,\mathbf{n}^{m+1,i} \r)^h_{\Gamma^{m+1,i}} = -A^{m+1,i}+A^0,
	\end{align}
\end{subequations}
where $\mathbf{n}^{m+1,i}$ is defined as in \eqref{ndef} for $\bX^{m+1,i}$, $L^{m+1,i}=L(\bX^{m+1,i})$ and $A^{m+1,i}=A(\bX^{m+1,i})$ are the perimeter and area of the polygon $\Gamma^{m+1,i}=\bX^{m+1,i}(\I)$, respectively. Note here on the left hand sides of  \eqref{Lag3:Newton} and \eqref{Lag4:Newton}, we used the variation formulae \eqref{First variation of L} and \eqref{First variation of A}, respectively for each polygon $\Gamma^{m+1,i}$ at the $i$-th iteration.
Then the iteration is updated via
\begin{align*}
	\bX^{m+1,i+1}=\bX^{m+1,i}+\bX^\delta ,&\qquad \kappa^{m+1,i+1}=\kappa^{m+1,i}+\kappa^\delta,\\
	\lambda^{m+1,i+1}=\lambda^{m+1,i}+\lambda^\delta ,&\qquad \eta^{m+1,i+1}=\eta^{m+1,i}+\eta^\delta,
\end{align*}
until
$$\max\{\|\bX^\delta\|_{L^\infty(\I)}, \|\kappa^\delta\|_{L^\infty(\I)}, |\lambda^\delta|, |\eta^\delta|\}\le \mathrm{tol},$$
 for some given tolerance $\mathrm{tol}$, and the numerical solution of the nonlinear system \eqref{Lageuler} is given by
\[
	\mathbf{X}^{m+1}=\mathbf{X}^{m+1,i+1},\quad \kappa^{m+1}=\kappa^{m+1,i+1},\quad
	\lambda^{m+1}=\lambda^{m+1,i+1}, \quad \eta^{m+1}=\eta^{m+1,i+1}.\]
Through the entire paper, the tolerance of Newton's iteration method is held constant at $\mathrm{tol}=10^{-10}$.

\smallskip

\begin{remark}\label{Rem:matrix}
We observe that the structure of the stiffness matrix for the system described by \eqref{LagNewton} is closely associated to $\kappa^{m+1}$. Specifically, the system \eqref{LagNewton} can be written as
	\[
	\begin{bmatrix}
		\mathbf{P} & \mathbf{Q} & \mathbf{a}_1 & \mathbf{a}_2\\
		\mathbf{R} & \mathbf{P}^T & 0 &0\\
		\mathbf{b}_1  & \mathbf{b}_2 & 0 &0\\
		\mathbf{c} & 0 & 0 & 0
	\end{bmatrix}
	\begin{bmatrix}
		\bX_\delta\\
		\bm{\kappa}_\delta\\
		\lambda^\delta \\
		\eta^\delta
	\end{bmatrix}=\begin{bmatrix}
		\mathbf{F}_1\\
		\mathbf{F}_2\\
		f_1\\
		f_2
	\end{bmatrix},
\]
where $\mathbf{P}\in \mathbb{R}^{N\times 2N}$, $\mathbf{Q}\in \mathbb{R}^{N\times N}$, $\mathbf{a}_1, \mathbf{a}_2\in \mathbb{R}^{N\times 1}$, $\mathbf{R}\in \mathbb{R}^{2N\times 2N}$, $\mathbf{b}_1,\mathbf{c}\in \mathbb{R}^{1\times 2N}$,  $\mathbf{b}_2\in \mathbb{R}^{1\times N}$, $\bX_\delta, \mathbf{F}_1\in\mathbb{R}^{2N\times 1}$, $\bm{\kappa}_\delta, \mathbf{F}_2\in\mathbb{R}^{N\times 1}$ and $f_1,f_2\in \mathbb{R}$.
Particularly, $\bm{\kappa}_\delta=\l(\kappa^\delta(\rho_1),\ldots,\kappa^\delta(\rho_N)\r)^T$.
Note that
$$\mathbf{a}_1=-\l(\l( \kappa^{m+1,i},\varphi_1\r)_{\Gamma^{m}}^h,
\ldots,\l( \kappa^{m+1,i},\varphi_N\r)_{\Gamma^{m}}^h\r)^T,$$
 where $\{\varphi_j\}_{j=1}^N$ are basis functions of $V^h$, and
\[\mathbf{a}_2=-\l(\l( 1,\varphi_1\r)_{\Gamma^{m}}^h,
\ldots,\l(1,\varphi_N\r)_{\Gamma^{m}}^h\r)^T,\]
thus $\mathbf{a}_1$ and $\mathbf{a}_2$ are linearly dependent
if $\kappa^{m+1,i}\in V^h$ is constant. This implies that Newton's iteration may struggle to converge within a few steps when $\kappa^m\in V^h$ is approximately constant. Consequently, it is not appropriate to employ the structure-preserving schemes with two Lagrange multipliers when the solution approaches equilibrium. In such instances, we transfer to an area-preserving scheme with a single Lagrange multiplier (cf. Section 4).
\end{remark}

\smallskip

\section{Perimeter-decreasing or area-preserving scheme with a single Lagrange multiplier}\label{sec:one:Lag}

 As highlighted in Remark \ref{Rem:matrix}, the formulation \eqref{Lag} may no longer be suitable for simulating the long-time evolution after reaching its equilibrium state. The challenge arises from the singularity of the stiffness matrix in Newton's iteration method when attempting to ensure two geometric structures simultaneously. To address this issue, we employ a simpler formulation that enforce only one geometric structure. Consequently, numerical methods for both first- and high-order PFEMs are capable of achieving stable long-time evolutions.

\subsection{Perimeter-decreasing schemes for long-time evolution}
We consider the following formulation with only one Lagrange multiplier for solving curve diffusion:
\begin{subequations}\label{LagPD}
\begin{align}
	\p_t\bX\cdot \bn
		&=\p_{ss}\kappa+\lambda(t) \kappa \label{Lag1:Long:PD} ,\\
		\kappa\bn &= -\p_{ss}\bX,\label{Lag2:Long:PD}\\
		\frac{\d L}{\d t}&=-\int_{\Gamma(t)}|\p_s\kappa|^2 \ \d s.\label{Lag3:Long:PD}
\end{align}	
\end{subequations}
By a similar discussion as presented in Proposition \ref{Prop:cont}, this formulation aligns with the BGN formulation if $\kappa\not\equiv 0$, suggesting that numerical methods based on this formulation may offer advantageous mesh quality during long-time evolution.

The schemes presented in Section 3 can be constructed in a similar way. We utilize the BDF$2$ time discretization as a representative example (denoted as the \textbf{PD-BDF$2$} scheme): Given $\Gamma^0$ and $\Gamma^{1}$, where $\Gamma^1=\bX^1(\I)$ and $\bX^1$ is the solution of the corresponding first-order Euler scheme for \eqref{LagPD}, for $m\ge 1$, determine $(\mathbf{X}^{m+1},\kappa^{m+1}, \lambda^{m+1})\in [V^h]^2\times V^h\times \R$ such that
\begin{subequations}\label{LagBDFkPD}
\begin{align}\label{Lag1:BDFk:Long:PD}
   \begin{split}
\l(\frac{\frac{3}{2}\bX^{m+1}-2\bX^m+\frac{1}{2}\bX^{m-1}}{\tau},\varphi^h \widetilde{\bn}^{m+1} \r)^h_{\widetilde{\Gamma}^{m+1}}&+\l( \p_s \kappa^{m+1},\p_s\varphi^h \r)_{\widetilde{\Gamma}^{m+1}} \\
-\lambda^{m+1}\l( \kappa^{m+1},\varphi^h \r)_{\widetilde{\Gamma}^{m+1}}^h& =0, \\
   \end{split}
\end{align}
\vspace{-10pt}
\begin{align}
 \l(\kappa^{m+1}\widetilde{\bn}^{m+1},\bm{\omega}^h\r)_{\widetilde{\Gamma}^{m+1}}^h-\l(\p_s \bX^{m+1},\p_s\bm{\omega}^h\r)_{\widetilde{\Gamma}^{m+1}}&=0,\label{Lag2:BDFk:Long:PD}\\
\frac{\frac{3}{2}L^{m+1}-2{L}^m+\frac{1}{2}L^{m-1}}{\tau}+\l( \p_s\kappa^{m+1},\p_s\kappa^{m+1}\r)_{\widetilde{\Gamma}^{m+1}}&=0.\label{Lag3:BDFk:Long:PD}
\end{align}
\end{subequations}
Similar to Theorem \ref{Thm:BDF}, \eqref{Lag3:BDFk:Long:PD} implies the perimeter-decreasing property.

\smallskip

\begin{theorem}\label{Thm:BDF:PD}
	Suppose $(\bX^{m+1},\kappa^{m+1}, \lambda^{m+1})\in [V^h]^2\times V^h\times\R$ is the solution of the PD-BDF$2$ scheme \eqref{LagBDFkPD}, then it holds
	\begin{equation*}
		L^{m+1}\le L^{m}\le \cdots\le L^0 ,\quad m\ge 0.
	\end{equation*}
\end{theorem}
	
In contrast to the SP-BDF2 scheme \eqref{LagBDFk}, this scheme can be effectively solved by Newton's iteration method even when the solution is close to the equilibrium. Note that this scheme is notable  from a numerical perspective for its unconditional energy stability and second-order accuracy.

\subsection{Area-preserving schemes for long-time evolution}
\label{aps}
It is also possible to consider the following formulation for curve diffusion:
\begin{subequations}\label{LagAP}
\begin{align}
	\p_t\bX\cdot \bn
		&=\p_{ss}\kappa+\eta(t)\label{Lag1:Long:AP} ,\\
		\kappa\bn &= -\p_{ss}\bX,\label{Lag2:Long:AP}\\
		\frac{\d A}{\d t}&=0.\label{Lag4:Long:AP}
\end{align}	
\end{subequations}	
Similar to the discussion in Proposition \ref{Prop:cont}, we can infer that \eqref{LagAP} coincides with \eqref{BGNmodel}, i.e., $\eta(\cdot, t)=0$, if $\int_{\Gamma(t)}\d s\neq 0$.

We consider the general BDF$k$ ($2\le k\le 6$) time discretization with prediction for system \eqref{Lag1:Long:AP}-\eqref{Lag4:BDFk:Long:AP} (denoted as the \textbf{AP-BDF$k$} scheme): Given $\Gamma^0,\ldots,\Gamma^{k-1}$ as suitable approximations of $\Gamma(t_0),\ldots$,\ $\Gamma(t_k)$,  for $m\ge k-1$, find $(\mathbf{X}^{m+1},\kappa^{m+1},\lambda^{m+1})\in [V^h]^2\times V^h \times\R$ such that
\begin{subequations}\label{LagBDFkAP}
\begin{align}\label{Lag1:BDFk:Long:AP}
   \begin{split}
   	  \Big(\frac{1}{\tau}\sum_{i=0}^k\delta_i\bX^{m+1-i},\varphi^h \widetilde{\bn}^{m+1} \Big)^h_{\widetilde{\Gamma}^{m+1}}+\l( \p_s \kappa^{m+1},\p_s\varphi^h \r)&_{\widetilde{\Gamma}^{m+1}} \\
\qquad\qquad\qquad\qquad\qquad\qquad\qquad\qquad -\eta^{m+1}\l( 1,\varphi^h \r)_{\widetilde{\Gamma}^{m+1}}^h &=0, \\
   \end{split}
\end{align}
\vspace{-9pt}
\begin{align}
 \l(\kappa^{m+1}\widetilde{\bn}^{m+1},\bm{\omega}^h\r)_{\widetilde{\Gamma}^{m+1}}^h-\l(\p_s \bX^{m+1},\p_s\bm{\omega}^h\r)_{\widetilde{\Gamma}^{m+1}}&=0,\label{Lag2:BDFk:Long:AP}\\
	\sum_{i=0}^k\delta_iA^{m+1-i}&=0.\label{Lag4:BDFk:Long:AP}
\end{align}
\end{subequations}
Here the coefficients of the method are given by
\begin{equation}\label{BDF:coe}
	\delta(\zeta)=\sum_{i=0}^k\delta_i \zeta^i=\sum_{\ell=1}^k\frac{1}{\ell}(1-\zeta)^\ell.
\end{equation}
Moreover, $\Gamma^0=\bX^0(\I),\ldots,\Gamma^{k-1}=\bX^k(\I)$ are given by the solutions of lower-order BDF$k$ methods with smaller time step size such that the errors of the initial $k$ curves are at $O(\tau^k)$, $\widetilde{\Gamma}^{m+1}$ are predicted by $(k-1)$th-order BDF$k$ method. Particularly, the first-order BDF$k$ method represents the corresponding backward Euler scheme, which coincides with \eqref{Lageuler} with $\lambda^{m+1}=0$ and without \eqref{Lag3:E}. We refer to \cite{Jiang24} for more details and implementations for BDF$k$ methods with $2\le k\le 4$, and refer to \cite{Hairer} for a general introduction to BDF methods.

\begin{theorem}\label{Thm:BDF:AP}
	The AP-BDF$k$ scheme \eqref{LagBDFkAP} is area-preserving:
\[
	A^{m+1}=A^m=\cdots\equiv A^0,\quad m\ge 0.
	\]
\end{theorem}

\begin{proof}
We prove by induction for $1\le k\le 6$ and $0\le m\le T/\tau$. For $k=1$, the AP-BDF$k$ scheme is actually the backward Euler scheme \eqref{Lageuler} with $\eta^{m+1}=0$ and without \eqref{Lag3:E}. Thus \eqref{Lag4:E} gives $A^m=\cdots\equiv A^0$ directly.
Now suppose the conclusion is valid for the AP-BDF$k$ schemes with order $k\le r<6$, we show the conclusion for the AP-BDF$k$ scheme with order $r+1$ by induction for $0\le m\le T/\tau$.
By the setting of the initial curves $\Gamma^1$, $\ldots$, $\Gamma^{r}$ and the induction, for the $(r+1)$th-order BDF$k$ scheme, it holds
\[A^{r}=A^{r-1}=\ldots=A^0.\]
For the AP-BDF$k$ scheme with order $r+1$, suppose
\[A^m=A^{m-1}=\ldots\equiv A^0,\quad r\le m<T/\tau,\]
we show that $A^{m+1}=A^0$.
The equation \eqref{Lag4:BDFk:Long:AP} for $k=r+1$ yields that
\begin{align*}
0&=\sum_{i=0}^{r+1}\delta_iA^{m+1-i}=\delta_0A^{m+1}+
  \sum_{i=1}^{r+1}\delta_iA^{m+1-i}\\
  &=\delta_0A^{m+1}+A^0 \sum_{i=1}^{r+1}\delta_i=\delta_0\l(A^{m+1}-A^0\r),
  \end{align*}
  where we used the fact $\sum\limits_{i=0}^{k}\delta_i=0$ by plugging $\zeta=1$ in \eqref{BDF:coe}.
  Thus we complete the proof by noticing $\delta_0>0$ and induction. \end{proof}

\smallskip

These schemes can also be solved efficiently  using Newton's iteration method for long-time evolution. Although proving a perimeter-decreasing property for \eqref{LagBDFkAP} seems challenging, the numerical behavior of \eqref{LagBDFkAP} and \eqref{LagBDFk} does not significantly differ (cf. Section 5).

\subsection{Modification for SP-type schemes by AP-type schemes}

As previously discussed in Remark \ref{Rem:matrix} and inspired by recent developments in the modified Lagrange multiplier approach for the phase field model \cite{Cheng2024}, for the SP-type schemes proposed in Section \ref{sec:SP}, we use a threshold $\gamma\ll 1$ to determine if equilibrium has been reached. Taking the SP-Euler scheme \eqref{Lageuler} as an example, inspired by \eqref{Lag3} or \eqref{Lag3:E}, if
\[\delta L^{m}:=\frac{L^{m+1}-L^m}{\tau}>\gamma,\]
we continue the original SP-Euler scheme \eqref{Lageuler} solved by Newton's iteration method; otherwise, we set $\lambda^{m+1}=0$ and omit the third equation \eqref{Lag3:E}. The resulting scheme, given by \eqref{Lag1:E}-\eqref{Lag2:E} and \eqref{Lag4:E}, belong to the AP-type schemes presented in Subsection \ref{aps}, which preserve area and facilitate  long-term evolution.

The complete modified algorithm is detailed in Algorithm \ref{Euler algorithm}. Numerical results (cf. Subsection \ref{sec:evolve}) demonstrate that equilibrium can be effectively achieved using Algorithm \ref{Euler algorithm}.  Furthermore, it is observed that the area remains constant and the perimeter decreases throughout the evolution. Similar modifications can be readily applied to the SP-CN scheme \eqref{LagCN} and the SP-BDF$2$ scheme \eqref{LagBDFk}.
In the following simulations, we consistently set the modification threshold to $\gamma=50\tau$.

\begin{algorithm}
\textbf{Complete algorithm for SP-Euler scheme  \eqref{Lageuler} with modifications}
	\label{Euler algorithm}
\begin{algorithmic}[1]
 \REQUIRE An initial curve $\Gamma(0)$ approximated by a polygon $\Gamma^0$ with $N$ vertices, described by $\bX^0\in [V^h]^2$ , time step $\tau$, terminate time $T$ satisfying $T/\tau\in\mathbb{N}$, a tolerance $\mathrm{tol}$ of Newton's iteration  method and a small threshold $\gamma$. The initial curvature $\kappa^0\in V^h$ is computed by solving \eqref{Lag2:semi} via a least square method \cite{BGN07A,Jiang23}, and set $\lambda^0=0$, $\eta^0=0$.
 \WHILE{$m < T/\tau$,}
 \STATE Compute $\bX^{m+1}$ by using SP-Euler scheme \eqref{Lageuler}, solved by Newton's iteration method \eqref{LagNewton} with tolerance $\mathrm{tol}$, $m = m + 1$.
 \STATE Calculate $\delta L^{m}=\frac{L^{m+1}-L^m}{\tau}$.
 \STATE \textbf{if} $|\delta L^m|\le \gamma$, \textbf{then} set $\lambda^{m+1}=0$ and
  \WHILE{$m < T/\tau$,}
  \STATE  Compute $\bX^{m+1}$ by using AP-Euler scheme \eqref{LagBDFkAP} for $k=1$, $m = m + 1$ .
   \ENDWHILE
 \ENDWHILE
 \end{algorithmic}
\end{algorithm}

\smallskip

\section{Numerical results}

In this section, we conduct the convergence order tests for our proposed structure-preserving schemes. We then illustrate the structure-preserving properties by tracking the evolution of geometric quantities for various types of initial curves.
Finally, we present some numerical experiments for the PD-BDF2 scheme and the AP-BDF$k$ schemes.

\subsection{Cauchy test for the convergence order}\label{sec:EOC}

In this subsection, we present a detailed Cauchy-type convergence order test for the SP-Euler scheme \eqref{Lageuler}, the SP-CN scheme \eqref{LagCN} and the SP-BDF$2$ scheme \eqref{LagBDFk}. The Cauchy-type convergence order test has been applied to first-order temporal numerical schemes for geometric flows \cite{BGN07A,BGN07B,DDE2005} and to second-order temporal numerical schemes for phase field models \cite{Shen2012,Hu2009}.


For the convergence test, we use the ellipse initial curve parametrized by
\[
\bX_0(\rho)=(2\cos(2\pi\rho),\sin(2\pi\rho)),\quad \rho\in \mathbb{I}.
\]
We take the terminal time as $T=0.25$, which is significantly distant from equilibrium; consequently, the modification step in Algorithm \ref{Euler algorithm} will not be executed. We then choose the Cauchy-type refinement path with $\tau=0.05h$. The error $\mathcal{E}_{\mathcal{M}}$ and the convergence order are given by
\begin{equation}
\cE_{\mathcal{M}}(T,\tau_1,\tau_2):= \mathcal{M}(\Gamma^{T/\tau_1}_{h_1}, \Gamma^{T/\tau_2}_{h_2}),\quad \text{Order}=\log\Big(\frac{\cE_{\mathcal{M}}(T,\tau_1,\tau_2)}{\cE_{\mathcal{M}}
(T,\tau_2,\tau_3)} \Big)\Big/ \log\l(\frac{\tau_1}{\tau_2} \r),
\end{equation}
respectively, where $\Gamma^{T/\tau}_h=\bX^{T/\tau}_h(\mathbb{I})$, $(\tau_j, h_j)$ satisfies the relation $\tau_j=0.05h_j$, and   $\mathcal{M}$ is the manifold distance \cite{Zhao2021,Jiang23} defined as
\[
	\mathcal{M}\l(\Gamma_1,\Gamma_2\r)
	:= |(\Omega_1\setminus\Omega_2)\cup (\Omega_2\setminus\Omega_1) | =|\Omega_1 |+|\Omega_2 |-2 |\Omega_1\cap \Omega_2 |,
\]
where $\Omega_1$ and $\Omega_2$ represent the regions enclosed by $\Gamma_1$ and $\Gamma_2$, respectively, and $|\Omega|$ denotes the area of $\Omega$. Additionally, it is intriguing to explore the numerical convergence of the Lagrange multipliers $\lambda$ and $\eta$. As both $\lambda$ and $\eta$ represent approximations of zero at the continuous model, we naturally compute the error and determine the order of convergence as
\begin{align}
 \mathcal{E}^{\lambda}_{\mathrm{Abs}}(T,\tau_1):=|\lambda^{T/\tau_1}-0|,\quad \text{Order}=\log\l(\frac{\cE^{\lambda}_{\mathrm{Abs}}(T,\tau_1)}{\cE^{\lambda}_{\mathrm{Abs}}
(T,\tau_2)} \r)\Big/ \log \l(\frac{\tau_1}{\tau_2} \r),\end{align}
and similar definition applies to the other Lagrange multiplier $\eta$.

Figure \ref{Fig:EOC} showcases the corresponding numerical outcomes for the first-order SP-Euler scheme \eqref{Lageuler}, the second-order schemes \eqref{LagCN} and \eqref{LagBDFk}. As expected, the SP-Euler scheme and the Crank-Nicolson and BDF$2$ time discretization methods exhibit robust linear and quadratic convergence in terms of manifold distance,  respectively. Additionally, the Lagrange multipliers associated with the BDF$2$ method also converge quadratically.

\begin{figure}[h!]
\hspace{125pt}
\begin{minipage}{0.32\linewidth}
\centerline{\includegraphics[width=5.0in,height=1.8in]{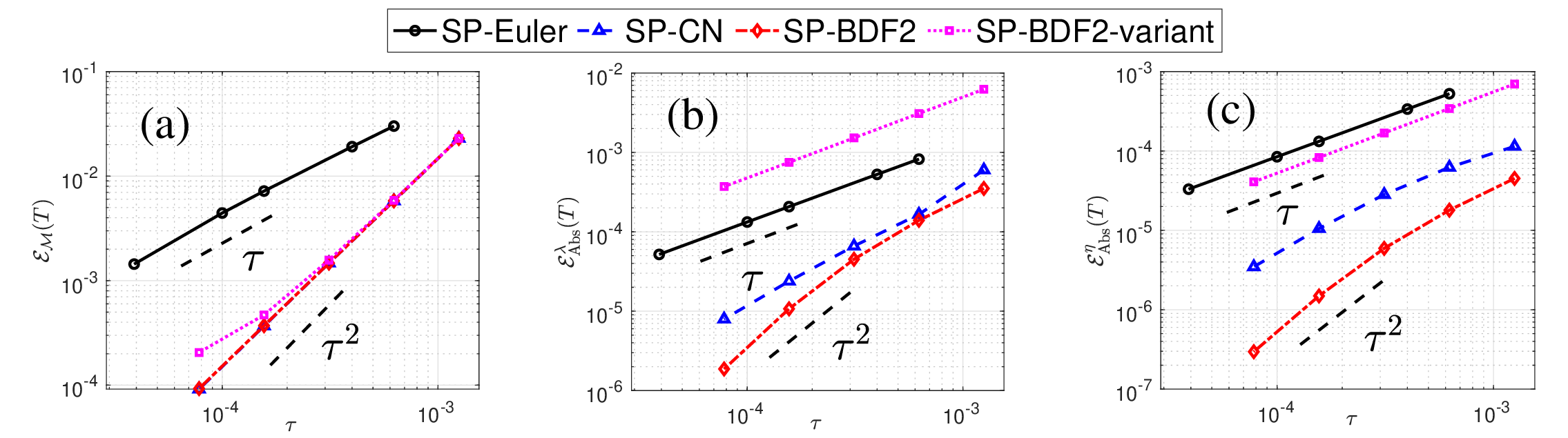}}
\end{minipage}
\caption{Log-log plot of the numerical errors with an ellipse as its initial shape at time $T=0.25$. For the SP-Euler scheme  \eqref{Lageuler}, the Cauchy path is chosen as $\tau=h^2$. For the SP-CN scheme \eqref{LagCN}, the SP-BDF2 scheme \eqref{LagBDFk} and the SP-BDF2-variant scheme with replacement \eqref{Lag3:BDFk} replaced by \eqref{Lag3:BDFk:change}, the Cauchy paths are chosen as  $\tau=0.05h$.}
\label{Fig:EOC}
\end{figure}

As discussed in Remark \ref{Change third equation}, one might wonder what happens if the third equation in the BDF2 scheme \eqref{LagBDFk} is replaced by \eqref{Lag3:BDFk:change} (denoted by SP-BDF2-variant). Figure \ref{Fig:EOC} illustrates the corresponding results, from which we observe that the convergence order rapidly degrades to first order. This highlights the necessity of consistent time discretization across all equations in \eqref{LagBDFk} at the same time level, as previously discussed in \cite{Jiang24}. By the way, the equations \eqref{Lag4:E} and \eqref{Lag4:BDFk} make no differences since both of them are equivalent to the equation $A^{m+1}=A^0$, as shown in Theorem \ref{Thm:BDF}.

\subsection{Structure-preserving properties}\label{sec:evolve}
In this subsection we conduct simulations of curve evolution and their geometric quantities for various initial curves. The geometric quantities considered in this subsections are the normalized perimeter and the relative area loss defined by
\[
L(t)/L(0)|_{t=t_m}=L^m/L^0,\qquad \Delta A(t)|_{t=t_m}=(A^m-A^0)/A^0.
\]
We also monitor the evolution of two Lagrange multipliers $$\lambda(t)|_{t=t_m}=\lambda^m,\quad \eta(t)|_{t=t_m}=\eta^m,$$
and the evolution of mesh ratio function, which is defined by
\[
\Psi(t)|_{t=t_m}=\max_j|\mathbf{h}^m_j|/\min_j|\mathbf{h}^m_j|.
\]

For the initial curves, we take the following  benchmark examples considered in \cite{BGN07A,Jiang23,Jiang24,Mikula-Sevcovic,Bao-Zhao}:
\begin{itemize}
	\item An ellipse with semi-major axis $2$ and semi-minor axis $1$;
	\item Mikula-\v{S}ev\v{c}ovi\v{c}'s curve   \cite{Mikula-Sevcovic}, parametrized by
	\[
\bX_0(\rho)=(
	 \cos(2\pi \rho),
	\sin(\cos(2\pi\rho ))+\sin(2\pi \rho)(0.7+\sin(2\pi \rho)\sin^2(6\pi \rho))),\ \rho\in \mathbb{I}.
\]
This is a curve with highly oscillatory curvature;
	\item A rectangular with length $4$ and width $1$. This is a nonsmooth initial curve.
\end{itemize}

\begin{figure}[htpb]
\centering
\includegraphics[width=4.9in,height=1.55in]{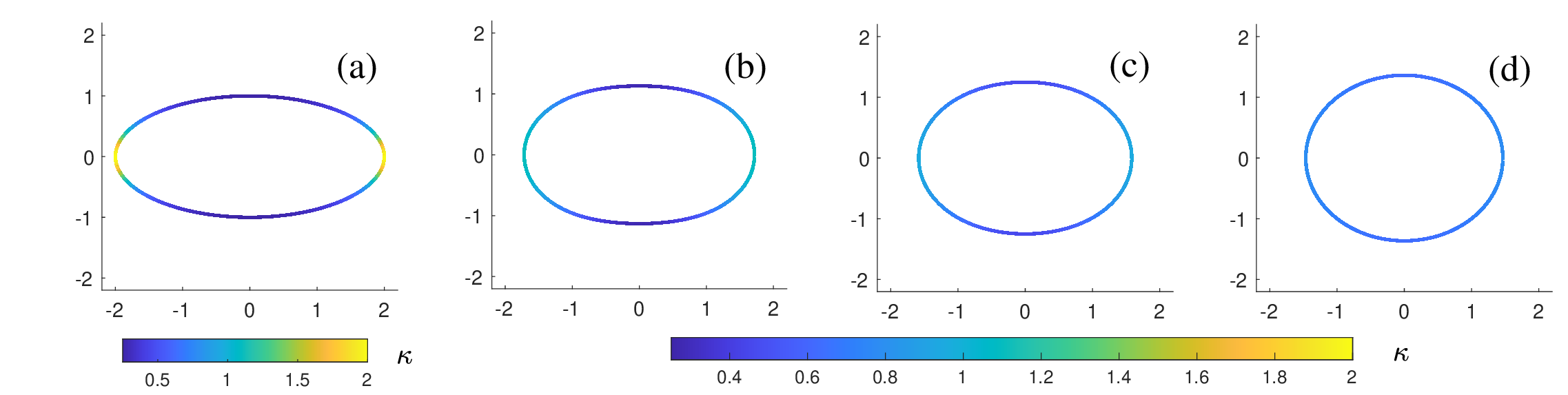}
\setlength{\abovecaptionskip}{2pt}
\caption{Snapshots of curve evolution with numerical curvature using the SP-BDF2 scheme \eqref{LagBDFk}, starting with an ellipse initial curve: (a) $t=0$, (b) $t=0.2$, (c) $t=0.4$, (d) $t=0.8$. The discretization parameters are chosen as $N=160$, $\tau=1/640$.}
\label{Fig:EVO_ell}
\end{figure}

\begin{figure}[h!]
\hspace{125pt}
\begin{minipage}{0.32\linewidth}
 \centerline{\includegraphics[width=4.8in,height=1.2in]{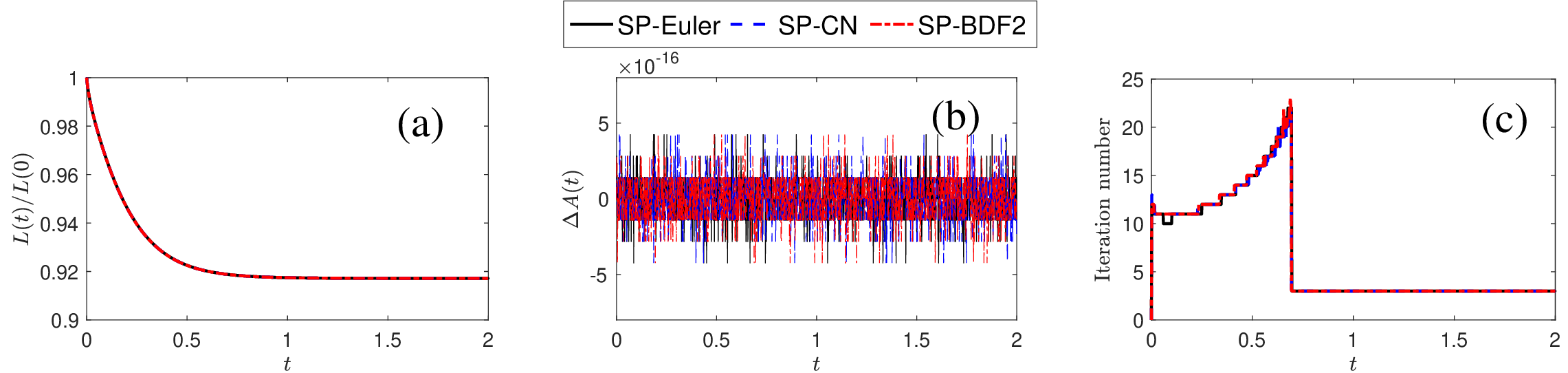}}
 \vspace{-7pt}
\centerline{\includegraphics[width=4.8in,height=1.2in]{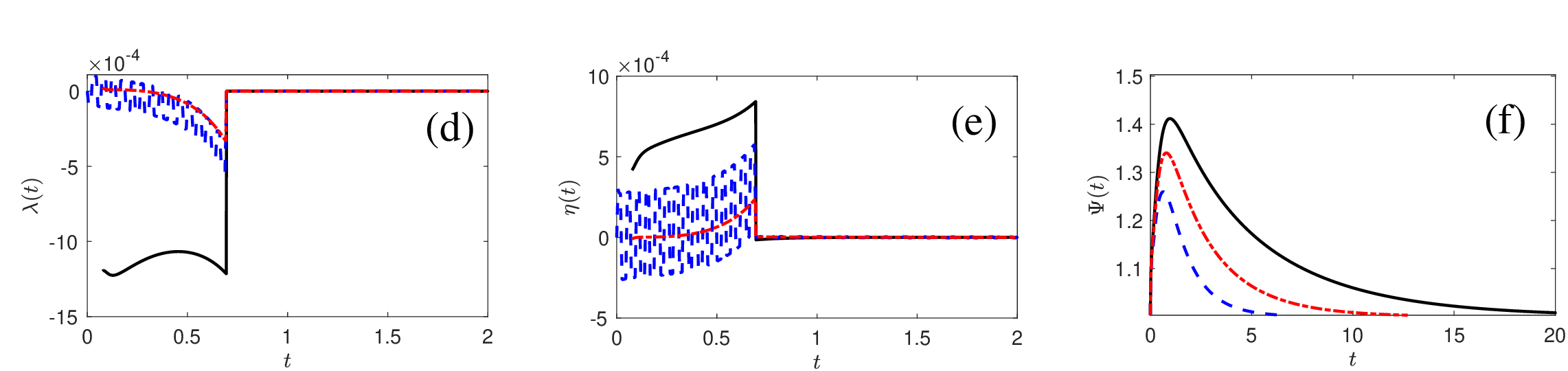}}
\end{minipage}
\caption{Evolution of geometric quantities and Lagrange multipliers as functions of time for three schemes, with the initial curve being an ellipse. (a) The normalized perimeter. (b) The relative area loss. (c) The iteration number of Newton's method. (d) The Lagrange multiplier $\lambda(t)$. (e) The Lagrange multiplier $\eta(t)$. (f) The mesh ratio function $\Psi(t)$. The discretization parameters are chosen as $N=160$, $\tau=1/640$.}
\label{Fig:GEO_ell}
\end{figure}

Figure \ref{Fig:EVO_ell} presents the evolution of an ellipse toward equilibrium using the SP-BDF2 scheme \eqref{LagBDFk}. Each snapshot is color-coded based on the approximation of curvature $\kappa^{m+1}\in V^h$. The SP-BDF2 scheme effectively simulates curve diffusion, as evidenced by the constant numerical curvature displayed in Figure \ref{Fig:EVO_ell} (d). Furthermore, Figure \ref{Fig:GEO_ell}  compares the performance of three structure-preserving schemes. As shown in Figure \ref{Fig:GEO_ell} (a) and (b), all schemes exhibit properties of perimeter reduction and area preservation (within machine precision), thereby corroborating the theoretical results presented in Theorems \ref{Thm:E}, \ref{Thm:CN} and \ref{Thm:BDF}. Figure \ref{Fig:GEO_ell} (c) reveals that the iteration number becomes larger when approaching equilibrium (approximately at $t=0.6$), as discussed in Remark \ref{Rem:matrix}. We emphasize that after achieving equilibrium, the modification procedure, as outlined in Step 5 of Algorithm \ref{Euler algorithm}, will be initiated. This procedure facilitates the transition of SP-type schemes employing two Lagrange multipliers to AP-type schemes that utilize a single Lagrange multiplier, as discussed in Subsection \ref{aps}. This transition  significantly decreases the number of iterations (cf. Figure \ref{Fig:GEO_ell} (c)) to solve the nonlinear system at each time step, thereby enhancing the efficiency of our algorithm for long-term evolution. As demonstrated in Subsection \ref{aps}, the AP-type schemes are area-preserving (cf. Figure \ref{Fig:GEO_ell} (b)). Moreover, Figure \ref{Fig:GEO_ell} (a) indicates that the perimeter remains almost constant after equilibrium, while Figure \ref{Fig:GEO_ell} (f) illustrates that the transition from SP-type to AP-type schemes substantially improves the mesh quality after equilibrium, as evidenced by their asymptotic equidistribution property.
Furthermore, Figure \ref{Fig:GEO_ell} (d) and (e) show that both Lagrange multipliers approach zero, suggesting that our formulation closely approximates the original BGN formulation.

\begin{figure}[htpb]
\centering
\includegraphics[width=4.9in,height=1.55in]{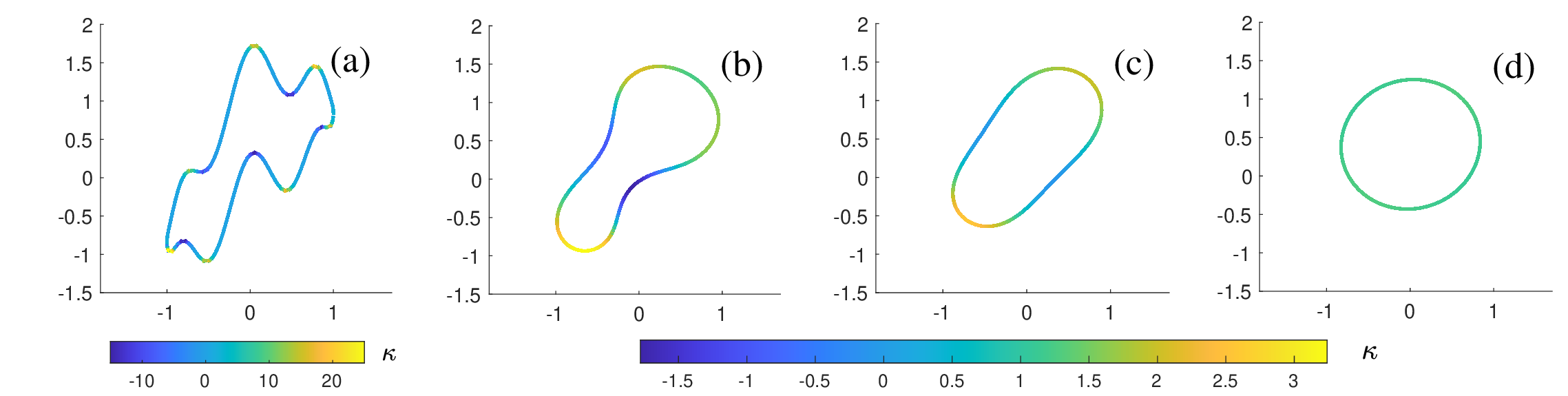}
\setlength{\abovecaptionskip}{2pt}
\caption{Snapshots of the curve evolution with its numerical curvature by the SP-BDF2 scheme \eqref{Lag1:BDFk}-\eqref{Lag4:BDFk}, starting with Mikula-\v{S}ev\v{c}ovi\v{c}'s curve. (a) $t=0$, (b) $t=0.005$, (c) $t=0.03$, (d) $t=0.15$. The parameters are chosen as $N=160$, $\tau=1/6400$.}
\label{Fig:EVO_Mikula}
\end{figure}

\begin{figure}[h!]
\hspace{125pt}
\begin{minipage}{0.32\linewidth}
 \centerline{\includegraphics[width=4.8in,height=1.2in]{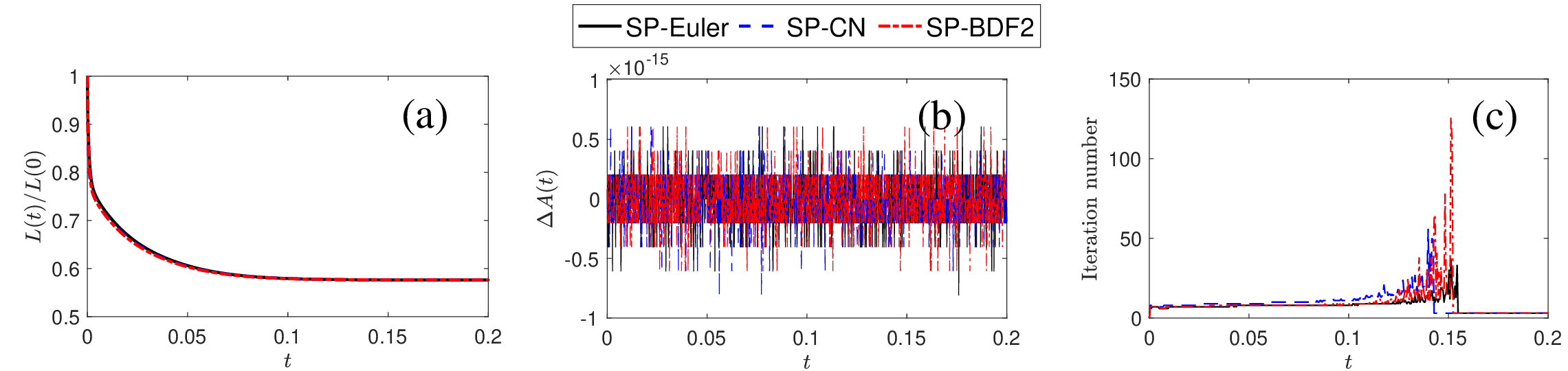}}
 \vspace{-7pt}
 \hspace{-1pt}\centerline{\includegraphics[width=4.8in,height=1.2in]{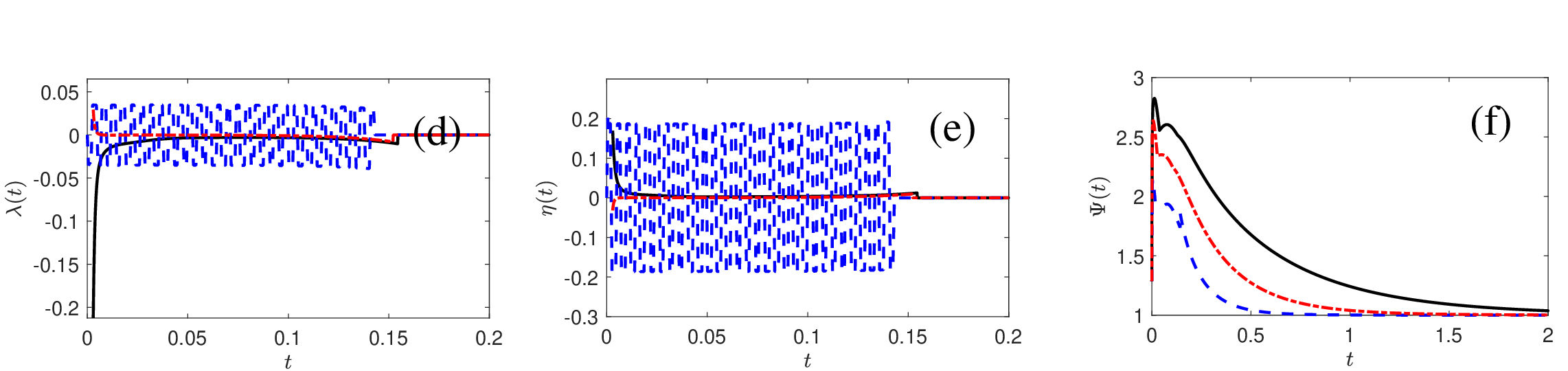}}
\end{minipage}
\caption{Evolution of geometric quantities and Lagrange multipliers as functions of time for three schemes, starting with Mikula-\v{S}ev\v{c}ovi\v{c}'s curve. (a) The normalized perimeter. (b) The relative area loss. (c) The iteration number of Newton's method. (d) The Lagrange multiplier $\lambda(t)$. (e) The Lagrange multiplier $\eta(t)$. (f) The mesh ratio function $\Psi(t)$. The parameters are chosen as $N=160$, $\tau=1/6400$.}
\label{Fig:GEO_Mikula}
\end{figure}

\vspace{0pt}

\begin{figure}[htpb]
\centering
\includegraphics[width=4.9in,height=1.55in]{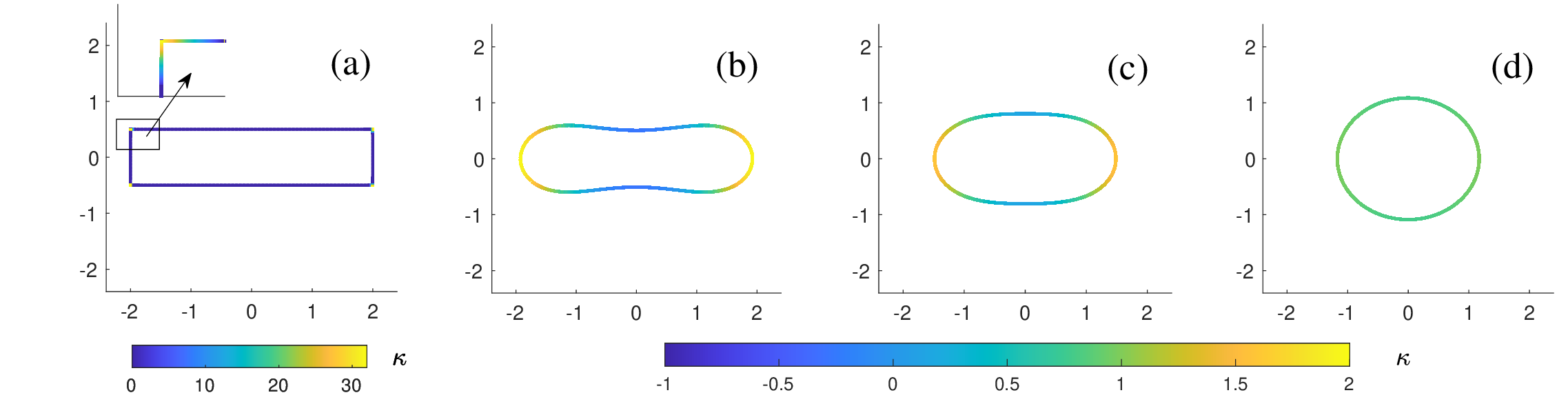}
\setlength{\abovecaptionskip}{2pt}
\caption{Snapshots of the curve evolution with its numerical curvature by the SP-BDF2 scheme \eqref{Lag1:BDFk}-\eqref{Lag4:BDFk}, starting with a rectangular initial curve. (a) $t=0$, (b) $t=0.04$, (c) $t=0.2$, (d) $t=0.5$. The parameters are chosen as $N=160$, $\tau=1/6400$.}
\label{Fig:EVO_rec}
\end{figure}

\begin{figure}[h!]
\hspace{125pt}
\begin{minipage}{0.32\linewidth}
 \centerline{\includegraphics[width=4.8in,height=1.2in]{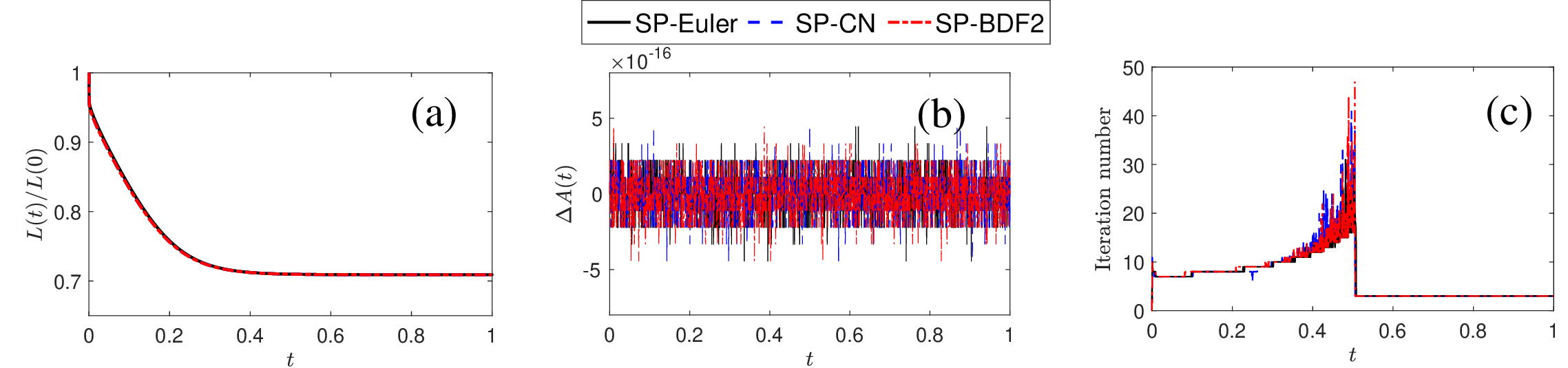}}
 \vspace{-7pt}
 \centerline{\includegraphics[width=4.8in,height=1.2in]{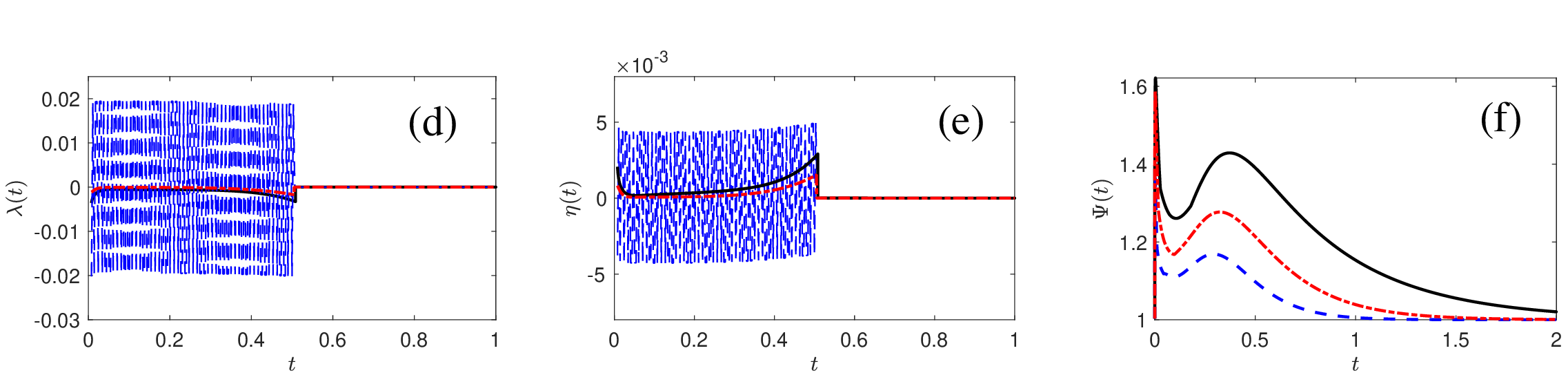}}
\end{minipage}
\caption{Evolution of geometric quantities and Lagrange multipliers as functions of time for three schemes, starting with a rectangular initial curve. (a) The normalized perimeter. (b) The relative area loss. (c) The iteration number of Newton's method. (d) The Lagrange multiplier $\lambda(t)$. (e) The Lagrange multiplier $\eta(t)$. (f) The mesh ratio function $\Psi(t)$. The parameters are chosen as $N=160$, $\tau=1/6400$.}
\label{Fig:GEO_rec}
\end{figure}

We then apply our methods to more general initial curves. Figures \ref{Fig:EVO_Mikula} and \ref{Fig:EVO_rec} show the evolution of Mikula-\v{S}ev\v{c}ovi\v{c}'s curve and a rectangle using the SP-BDF2 scheme \eqref{LagBDFk}, respectively. It is clear that both curves tend toward a circular equilibrium, with the numerical curvature converging to a constant value (as evidenced by Figure \ref{Fig:EVO_Mikula}(d) and Figure \ref{Fig:EVO_rec}(d)). Furthermore, Figure \ref{Fig:GEO_Mikula} and Figure \ref{Fig:GEO_rec} demonstrate that our algorithms effectively preserve the perimeter-decreasing and area-preserving properties and exhibit asymptotic equidistribution property when the AP-type schemes commence after equilibrium (see also Subsection \ref{single}).

\subsection{Long-term evolution for the formulation with a single Lagrange multiplier}
\label{single}
In this subsection, we conduct numerical experiments for the schemes presented in Section \ref{sec:one:Lag} that preserve only one specific geometric structure. The convergence results of the PD-BDF2 scheme \eqref{LagBDFkPD} and the AP-BDF$k$ ($k=2,3$) scheme \eqref{LagBDFkAP} are presented in Tables \ref{Tab:LD} and \ref{Tab:AP}, respectively. For the AP-BDF$k$ scheme with $k=3$, we choose the Cauchy-type refinement path $\tau=0.05h^{2/3}$.  Similar to the structure-preserving schemes with two Lagrange multipliers, one can observe the desired temporal accuracy for the numerical solution $\bX^m$ under the manifold distance.

\begin{table}[h!]\label{Tab:LD}
\centering
\renewcommand{\arraystretch}{1.4}
\def\temptablewidth{0.6\textwidth}
\vspace{-2pt}
\caption{Accuracy test for the PD-BDF2 scheme \eqref{LagBDFkPD} with ellipse as initial. The terminal time is chosen as $T=0.25$, and the Cauchy-type refinement path is chosen as $\tau=0.05h$.}
{\rule{\temptablewidth}{1pt}}
\begin{tabular*}{\temptablewidth}{@{\extracolsep{\fill}}ccc}
 $\tau$
 & $\mathcal{E}_{\mathcal{M}}(T)$&  Order  \\  \hline
1/200  & 2.28E-2  &  ---  \\   \hline
1/400  & 5.75E-3  & 1.99   \\   \hline
1/3200  & 1.45E-3  &  1.98   \\  \hline
1/6400   & 3.65E-4  &  1.99  \\   \hline
1/12800   & 9.11E-5 &  2.00 \\
 \end{tabular*}
{\rule{\temptablewidth}{1pt}}
\end{table}

\begin{table}[h!]\label{Tab:AP}
\centering
\renewcommand{\arraystretch}{1.4}
\def\temptablewidth{0.98\textwidth}
\vspace{-2pt}
\caption{Convergence order test for the AP-BDF$k$ scheme ($k=2,3$) \eqref{LagBDFkAP} with an ellipse initial shape. The terminal time is chosen as $T=0.25$, and the Cauchy-type refinement path is chosen as $\tau=0.05h$ for $k=2$ and $\tau=0.05h^{2/3}$ for $k=3$.}
{\rule{\temptablewidth}{1pt}}
\begin{tabular*}{\temptablewidth}{cccccc}
$\tau$
 &  $\mathcal{E}_{\mathcal{M}}$\, \text{for AP-BDF2 }  & Order& $\tau$&  $\mathcal{E}_{\mathcal{M}}$\, \text{for AP-BDF3 } &  Order   \\ \hline
1/800  & 2.29E-2  &  --- &  1/500 & 2.21E-3  &  ---  \\   \hline
1/1600  & 5.81E-3 &  1.98 &  1/720  & 6.94E-4  &  3.17   \\   \hline
1/3200  & 1.47E-3 &  1.98 &  1/1280 & 2.62E-4  &  3.17   \\  \hline
1/6400   & 3.71E-4  &   1.99  & 1/1620  & 1.13E-4  &  3.16  \\   \hline
1/12800   & 9.30E-5   & 2.00 & 1/2000  & 5.36E-5  & 3.15 \\
 \end{tabular*}
{\rule{\temptablewidth}{1pt}}
\end{table}

We conclude this subsection by presenting the evolution of geometric quantities for Mikula-\v{S}ev\v{c}ovi\v{c}'s curve. As previously mentioned, dynamic evolution is not included here due to its similarity with Figure \ref{Fig:EVO_Mikula}.

Figure \ref{Fig:GEO_one_Lag} illustrates the evolution of geometric quantities for the PD-BDF2 scheme \eqref{LagBDFkPD}, the AP-BDF2 and AP-BDF3 schemes \eqref{LagBDFkAP}. We can draw the following conclusions: (i) The PD-BDF2 scheme is perimeter-decreasing and the AP-BDF2 and AP-BDF3 schemes conserve the area, aligning with Theorem \ref{Thm:BDF:PD} and Theorem \ref{Thm:BDF:AP}, respectively; (ii) All the  schemes exhibit a robust long-term evolution, with the iteration number remaining low in comparison to the structure-preserving schemes with two Lagrange multipliers presented in Section 3; (iii) All schemes possess the long-time equidistribution property of mesh points, as evidenced by $\Psi(t)\rightarrow 1$ as $t\rightarrow \infty$. This is attributed to the alignment of the semi-discrete schemes with the BGN formulation as well as their inherent equidistribution property.

\begin{figure}[h!]
\hspace{125pt}
\begin{minipage}{0.32\linewidth}
 \centerline{\includegraphics[width=4.8in,height=1.2in]{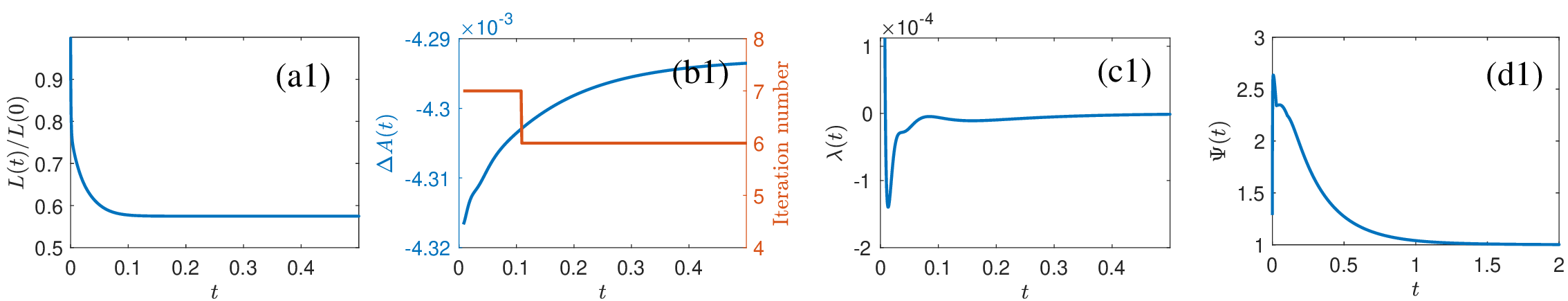}}
 \centerline{\includegraphics[width=4.8in,height=1.2in]{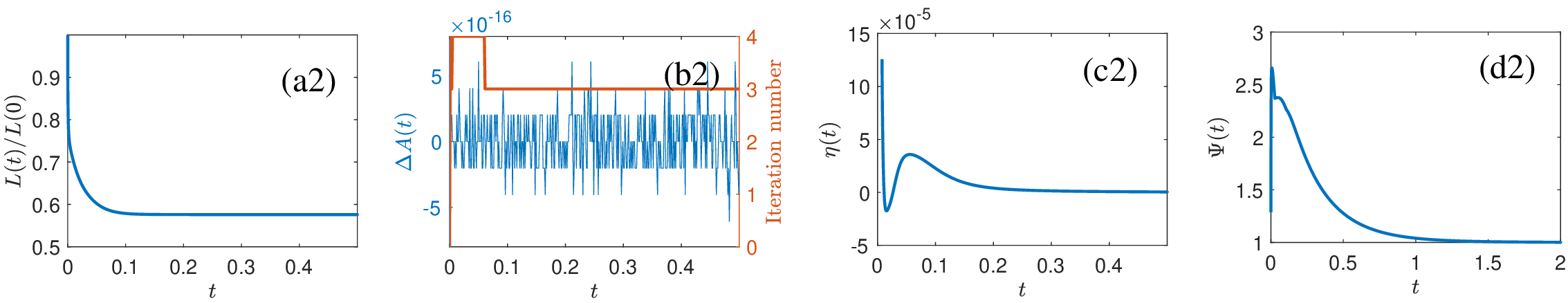}}
 \centerline{\includegraphics[width=4.8in,height=1.2in]{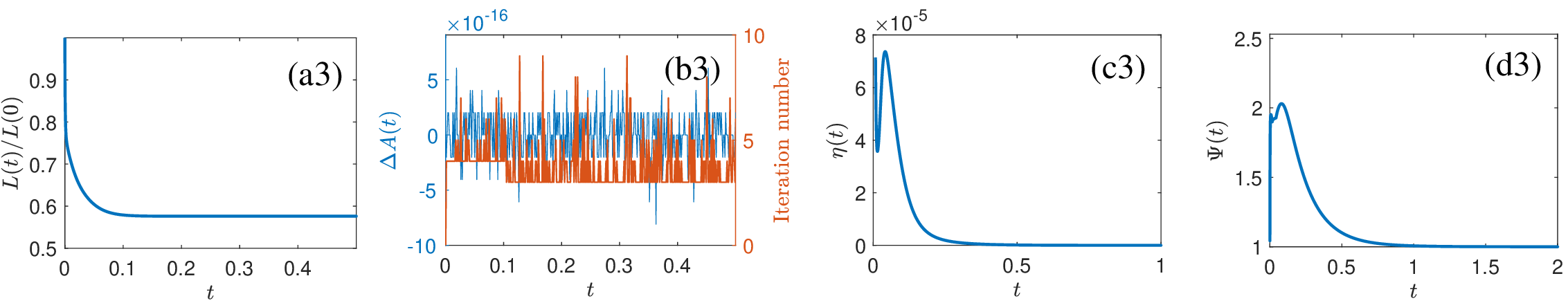}}
\end{minipage}
\caption{Evolution of geometric quantities and Lagrange multiplier for the PD-BDF2 scheme \eqref{LagPD} (first row), the AP-BDF2 scheme (second  row) and the AP-BDF3 scheme \eqref{LagAP} (third row), beginning with Mikula-\v{S}ev\v{c}ovi\v{c}'s curve.  The normalized perimeter: (a1)-(a3); The relative area loss and iteration number: (b1)-(b3); The Lagrange multiplier $\lambda(t)$: (c1) and $\eta(t)$: (c2), (c3); The mesh ratio function: (d1)-(d3). The discretization parameters are set to $N=160$, $\tau=1/6400$.}
\label{Fig:GEO_one_Lag}
\end{figure}

\section{Conclusion}

We developed several high-order structure-preserving schemes for solving curve diffusion using the Lagrange multiplier approach. Our methodology incorporates the classical BGN formulation with two supplementary scalar evolution equations that govern perimeter and area, facilitated by the introduction of two scalar Lagrange multipliers. By employing wither the Crank-Nicolson method or backward differentiation formula for time discretization and utilizing the linear finite element for spatial discretization, we proposed several fully discrete, structure-preserving schemes that demonstrate high-order accuracy in time, as measured by the manifold distance. Practically, our schemes can be efficiently solved using Newton's iteration prior to reaching equilibrium. Long-term evolution is feasible by simplifying to a single Lagrange multiplier method. Extensive numerical results have been reported to demonstrate the desired high-order accuracy and structure-preserving characteristics. Similar approaches can be extended to simulate various geometric flows, including mean curvature flow, three-dimensional surface diffusion and anisotropic geometric flow, which we intend to explore in future research.

\section*{Acknowledgments}
This work is supported by the Center of High Performance Computing, Tsinghua University and the Supercomputing Center of Wuhan University. Part of this work was done during Ganghui Zhang's visit to the Department of Mathematics at the University of Regensburg, Germany. Ganghui Zhang would like to express his sincere gratitude for their hospitality and support.


%
%

\end{document}